\documentclass[cmbright]{staauth}

\usepackage{moreverb}

\usepackage{natbib}
\usepackage{epsf,epsfig,subfigure}

\DeclareMathOperator{\pa}{pa}

\DeclareMathOperator{\nd}{nd}

\DeclareMathOperator{\var}{var}
\DeclareMathOperator{\cov}{cov}

\newcommand{\XVect}{X}
\newcommand{\noiseVect}{\epsilon}

\newcommand\independent{\protect\mathpalette{\protect\independenT}{\perp}}
\def\independenT#1#2{\mathrel{\rlap{$#1#2$}\mkern2mu{#1#2}}}

\newcommand\notindependent{\!\perp\!\!\!\!\not\perp\!}

\newtheorem{theorem}{Theorem}[section]
\newtheorem{lemma}[theorem]{Lemma}

\theoremstyle{definition}
\newtheorem{definition}[theorem]{Definition}
\newtheorem{example}[theorem]{Example}

\theoremstyle{remark}

\def\volumeyear{2012}
\runninghead{Garvesh Raskutti and Caroline Uhler}{Learning DAG models using sparsest permutations}%Learning DAGS from sparsest permutations}

\begin{document}

\title{Learning directed acyclic graph models based on sparsest permutations}

\author{Garvesh Raskutti\affil{a}, Caroline Uhler\affil{b}\corrauth}

\address{%
\affilnum{a}University of Wisconsin, Department of Statistics, Madison, WI, USA\\
\affilnum{b}Massachusetts Institute of Technology, Laboratory for Information and Decision Systems, and Institute for Data, Systems and Society, Cambridge, MA, USA
}

\corremail{cuhler@mit.edu}

\received{30 January 2018}
\accepted{15 March 2018}

\begin{abstract}
We consider the problem of learning a Bayesian network or directed acyclic graph model from observational data. A number of constraint-based, score-based and hybrid algorithms have been developed for this purpose. Statistical consistency guarantees of these algorithms rely on the faithfulness assumption, which has been shown to be restrictive especially for graphs with cycles in the skeleton. We here propose the sparsest permutation (SP) algorithm, showing that learning Bayesian networks is possible under strictly weaker assumptions than faithfulness. This comes at a computational price, thereby indicating a statistical-computational trade-off for causal inference algorithms. In the Gaussian noiseless setting we prove that the SP algorithm boils down to finding the permutation of the variables with the sparsest Cholesky decomposition of the inverse covariance matrix, which is equivalent to $\ell_0$-penalized maximum likelihood estimation. We end with a simulation study showing that in line with the proven stronger consistency guarantees the SP algorithm compares favorably to standard causal inference algorithms in terms of accuracy for a given sample size. 
\end{abstract}


\maketitle

\section{Introduction}

A fundamental goal in many scientific problems is to determine causal or directional relationships between variables in a system. A useful simplification is to assume that the causal structure is modeled by a \emph{directed acyclic graph} (\emph{DAG}) or a \emph{Bayesian network}. Given a DAG $G := ([p], E)$ with vertices $[p] := \{1,2\dots,p\}$ and directed edges $E$, a Bayesian network associates to each node $i\in[p]$ a random variable $X_i$. Missing edges in $G$ entail conditional independence (CI) relations, namely %typical of cause-effect relationships:
$
X_i \independent X_{\nd(i)\backslash\pa(i)} \,\mid\, X_{\pa(i)},
$
where $\nd(i)$ and $\pa(i)$ respectively denote the \emph{nondesendants} and \emph{parents} of vertex $i$ in $G$.   
A joint probability distribution $\mathbb{P}$ over $[p]$ is said to satisfy the \emph{Markov property} (or \emph{be Markov}) with respect to $G$ if it entails these CI relations. In fact, the complete set of CI relations implied by the Markov property can be read off from the DAG by the so-called \emph{d-separation} statements that hold in $G$~\citep[Section 3.2.2]{Lauritzen}.

Given $n$ i.i.d.~samples from a distribution $\mathbb{P}$ that is Markov to a DAG $G^*$, a fundamental problem in causal inference is to infer $G^*$. Unfortunately, this problem is illposed, since several DAGs can encode the same set of CI relations. Such DAGs are said to be \emph{Markov equivalent}. %\citet{VermaPeal90} showed that the \emph{Markov equivalence class} $\mathcal{M}(G^*)$ is uniquely defined by the \emph{skeleton} (i.e.,~the undirected edges of $G^*$) and the \emph{v-structures} (i.e., induced subgraphs of the form $j\to k\leftarrow \ell$). Thus, our goal is to identify $\mathcal{M}(G^*)$, or equivalently, the skeleton and the v-structures of $G^*$. 
Thus, our goal is to identify the \emph{Markov equivalence class} denoted $\mathcal{M}(G^*)$. 

A number of algorithms have been developed for this purpose and they can broadly be classified into \emph{constraint-based}, \emph{score-based}, and \emph{hybrid} methods. Constraint-based methods such as the prominent \emph{SGS}~\citep{GSSK87} and \emph{PC}~\citep{spirtes:00} algorithms treat causal inference as a constraint-satisfaction problem with the constraints being CI relations. On the other hand, score-based approaches assign a score function such as the Bayesian Information Criterion (BIC) to each DAG and optimize the score. Since searching over the space of DAGs is in general NP-complete~\citep{chickering:95}, the optimization is usually performed greedily such as in the prominent \emph{Greedy Equivalence Search} (\emph{GES})~\citep{chick02}. Finally, hybrid methods have also been considered; \emph{Max-Min Hill-Climbing} (\emph{MMHC})~\citep{Tsamardinos06}, for example, alternates between constraint- and score-based updates.

A minimal requirement for causal inference algorithms is (\emph{pointwise}) \emph{consistency}, i.e.~outputting the correct Markov equivalence class $\mathcal{M}(G^*)$ given %$n\to\infty$ i.i.d.~samples from 
a distribution $\mathbb{P}$ that is Markov to $G^*$. Consistency has been proven for the SGS and PC algorithms~\citep{spirtes:00} as well as for GES~\citep{chick02} under the so-called \emph{faithfulness} assumption.

\begin{definition}\citep{spirtes:00}
	\label{AssFaith}
	A distribution $\mathbb{P}$ is \emph{faithful} to a DAG $G$ if for any $j,k\in [p]$ and any $S\subset [p]\setminus\{j,k\}$,
	$$k \textrm{ is d-separated from } j \textrm{ given } S \qquad\Longleftrightarrow\qquad X_j\independent X_k \mid X_S.$$
\end{definition}

The faithfulness assumption requires that $\emph{all}$ CI relations of $\mathbb{P}$ are encoded in $G^*$. Unfortunately, this assumption is very sensitive to hypothesis testing errors for inferring CI relations from data; almost-violations are frequent especially in graphs with cycles in the skeleton~\citep{UhlerRasBuYu}. 
A number of attempts have been made to modify constraint-based algorithms to adjust for weaker conditions than faithfulness (e.g.,~\citet{Lemiere,RamseySpirtesZhang, ZhangSpirtes08, SpirtesZhang13}), always at a computational cost~\citep{Spirtes_3faces}. The weakest known sufficient conditions for identifying $\mathcal{M}(G^*)$ are given by the \emph{restricted-faithfulness} assumption.

\begin{definition} \citep{RamseySpirtesZhang}
	\label{RestFaith}
	A distribution $\mathbb{P}$ satisfies the \emph{restricted-faithfulness assumption} with respect to a DAG $G$ if it is Markov to $G$ and the following two conditions hold:
	\vspace{-0.2cm}
	\begin{enumerate}
%		\item[(i)] \emph{adjacency-faithfulness} \citep{RamseySpirtesZhang}: for all $(j, k) \in E$,
%		$$X_j \notindependent X_k \mid X_S, \quad\textrm{ for all } S \subset V\setminus\{j,k\};$$
%		\item[(ii)] \emph{orientation-faithfulness} \citep{RamseySpirtesZhang}: for all triples $(j, k, \ell)$ with skeleton $j-k-\ell$, %and all subsets $S\subset V\setminus \{j,k\}$ such that $j$ is d-connected to $k$ given $S$,
%		$$X_j \notindependent X_k \mid X_S, \quad\textrm{ for all } S \subset V\setminus\{j,k\} \textrm{ such that $j$ is d-connected to $k$ given $S$}.$$
		\item[(i)] \emph{adjacency-faithfulness}: for all $(j, k) \in E$ and all subsets $S \subset [p]\setminus\{j,k\}$ it holds that $X_j \notindependent X_k \mid X_S$;
\item[(ii)] \emph{orientation-faithfulness}: for all triples $(j, k, \ell)$ with skeleton $j-\ell-k$ and all subsets $S\subset [p]\setminus \{j,k\}$ such that $j$ is d-connected to $k$ given $S$ it holds that
$X_j \notindependent X_k \mid X_S.$
	\end{enumerate}
\end{definition}

A classic result~\citep{VermaPeal90} states that a Markov equivalence class $\mathcal{M}(G)$ is uniquely defined by the \emph{skeleton} (i.e.,~the undirected edges of $G$) and the \emph{v-structures} (i.e., induced subgraphs of the form $j\to k\leftarrow \ell$). The adjacency faithfulness assumption ensures recovery of the skeleton, while orientation faithfulness guarantees the correct orientation of edges in v-structures. In this paper, we provide an algorithm showing that it is possible to identify $\mathcal{M}(G^*)$ under strictly weaker conditions than restricted-faithfulness.

In Section~\ref{SecMain}, we present our \emph{Sparsest Permutation} (\emph{SP}) algorithm and its consistency guarantees. In Section~\ref{SecNoisy}, we consider the effects of hypothesis testing and provide uniform consistency guarantees for the SP algorithm. While the SP algorithm is non-parametric and can be applied for any distribution in which CI tests can be constructed, in Section~\ref{SecGaussian}, we consider the noiseless Gaussian setting. We use a result due to~\citet{Pourahmadi99} to prove that in this setting the SP algorithm boils down to finding the sparsest Cholesky factorization of the inverse covariance matrix and is equivalent to the $\ell_0$-penalized maximum likelihood approach described by~\citet{geerpb12}. Finally, Section~\ref{SecSimulations} provides a simulation study showing that the SP algorithm compares favorably in terms of accuracy for a given sample size to standard causal inference algorithms, in line with the proven stronger consistency guarantees. %We end with a short discussion in 

\section{Sparsest permutation (SP) algorithm}
\label{SecMain}

In this section, we introduce the SP algorithm and prove pointwise consistency of this algorithm under strictly weaker conditions than restricted-faithfulness. The SP algorithm is a hybrid method, which uses conditional independence relations to build a DAG for each ordering of the nodes and then uses the number of edges as score to rank the DAGs. More precisely, a permutation $\pi$ of the vertices $[p]$ is associated to a DAG $G_{\pi} = ([p], E_{\pi})$ through $\mathbb{P}$ as follows: %Let $j<k$, then
\begin{equation*}
(\pi(j), \pi(k)) \in E_{\pi}\quad\Longleftrightarrow\quad j<k \textrm{ and } X_{\pi(j)} \notindependent X_{\pi(k)} \mid X_{\{\pi(1),\pi(2),..., \pi(k-1)\}\setminus\{\pi(j)\}} \textrm{ in } \mathbb{P}.  %\;\textrm{ where }\; S = \{\pi(1),\pi(2),..., \pi(k-1)\}\setminus\{\pi(j)\}.
\end{equation*}
\citet[Theorem 9 on p.~119]{Pearl88} and also~\citet{Verma86} showed that for any positive measure $\mathbb{P}$ and any permutation $\pi$, $\mathbb{P}$ is Markov to $G_\pi$ and satisfies the \emph{minimality assumption}, meaning that there is no proper sub-DAG %(with the same vertex set and a strict subset of the directed edges) 
of $G_\pi$ that satisfies the Markov property. A DAG $G_{\pi}$ is therefore also referred to as a \emph{minimal I-MAP} (independence map); see~\citet{Pearl88}. In this paper, we propose to select the permutation $\pi$ that yields the smallest number of edges, i.e.,~the \emph{sparsest permutation} (SP), also referred to as the \emph{minimal-edge I-MAP}~\citep{geerpb12}. Denoting by $|G|$ the number of edges in $G$, the SP algorithm is hence defined as follows:
\begin{enumerate}
	\item[(1)] For all permutations $\pi$ of the vertices $[p]$ construct $G_\pi$ and let $s_\pi := |G_\pi|$.
	%\item[(2)] Choose the set of permutations $\{\pi^*\}$ for which $s_{\pi^*}$ is minimal amongst all permutations.
	\item[(2)] Output $G_{\pi^*}$ for all $\pi^*$ such that $s_{\pi^*}$ is minimal amongst all permutations. 
\end{enumerate}

This algorithm is an extension of the \emph{boundary-strata method}~\citep[p.~119]{Pearl88}, where the search is performed over all permutations. Permutation-based approaches for causal inference are not new~\citep{B92,SV93,TeyssierKoller}. However, none of these methods have consistency guarantees. Denoting by $\{G_{SP}\}$ the set of outputted graphs of the SP algorithm, the following result justifies choosing the sparsest permutation.

\begin{lemma}
	\label{LemSparsity}
	Let $\mathbb{P}$ be Markov to the DAG $G^*$ and let $S(G^*)$ denote its skeleton. Then as $n \rightarrow \infty\,$ it holds for all DAGs $\,G_{SP}$ that

%	\begin{enumerate}
%		\item[(a)] $|G_{SP}|\leq |G^*|$;
%		\item[(b)] if $G^*$ satisfies the adjacency-faithfulness assumption with respect to $\mathbb{P}$, then $S(G_{SP}) = S(G^*)$;
%		\item[(c)] if $G^*$ satisfies the restricted-faithfulness assumption with respect to $\mathbb{P}$, then $G_{SP} \in \mathcal{M}(G^*).$
%	\end{enumerate}
(a) $|G_{SP}|\leq |G^*|$;

(b) if $G^*$ satisfies the adjacency-faithfulness assumption with respect to $\mathbb{P}$, then $S(G_{SP}) = S(G^*)$;

(c) if $G^*$ satisfies the restricted-faithfulness assumption with respect to $\mathbb{P}$, then $G_{SP} \in \mathcal{M}(G^*).$
\end{lemma}

\begin{proof}
	(a) \; Let $\pi^*$ denote a permutation that is consistent with the partial order defined by $G^*$. Let $E^*$ denote the edges in $G^*$ and let $G_{\pi^*}$ denote the DAG resulting from applying the SP algorithm to the ordering $\pi^*$. Note that two nodes $\pi^*(i)$ and $\pi^*(j)$ with $i<j$ and $(\pi^*(i), \pi^*(j))\notin E^*$ are d-separated in $G^*$ by the set $S=\{\pi^*(1),\dots ,\pi^*(j-1)\}\setminus\{\pi^*(i)\}$. Then, as a consequence of the Markov property, the distribution $\mathbb{P}$ satisfies $X_{\pi^*(i)}\independent X_{\pi^*(j)} \mid X_S$ and hence, $(\pi^*(i), \pi^*(j))\notin E_{\pi^*}$. This shows that $|G_{\pi^*}| \leq |G^*|$ and together with $|G_{SP}|\leq|G_{\pi^*}|$ completes the proof.
%	\vspace{0.1cm}
	
	(b) \; As a consequence of (a), to show equality of the skeleta it suffices to prove $S(G^*)\subset S(G_{\pi})$ for all permutations $\pi$. Under adjacency-faithfulness, $X_j \notindependent X_k \mid X_S$ for all edges $(X_j, X_k) \in E^*$ and for all subsets $S \subset [p] \setminus\{j,k\}$. Then by the definition of $G_\pi$ any edge in $G^*$ must be an edge in $G_\pi$ and hence $S(G^*) \subset S(G_\pi)$ for all permutations $\pi$.
%	\vspace{0.1cm}
	
	(c) \; Since restricted-faithfulness entails adjacency-faithfulness, $S(G_{SP}) = S(G^*)$. %In the following we prove that $G_{SP} \in\mathcal{M}(G^*)$ as a consequence of the orientation-faithfulness assumption. 
	Let $(j,k,\ell)$ be a triple with skeleton $j-\ell-k$ in $G^*$. We consider two cases, first the case when $(j,k,\ell)$ is a v-structure. Then $X_j\notindependent X_k \mid X_S$ for all $S\subset [p]\setminus\{j,k\}$ such that $\ell\in S$. Let $\pi_{SP}$ be a permutation giving rise to $G_{SP}$ and let $a,b,c\in [p]$ such that $\pi_{SP}(a)=j$, $\pi_{SP}(b)=k$, $\pi_{SP}(c)=\ell$. Note that $(\pi_{SP}(a), \pi_{SP}(b))\notin E^*$ and hence also not an edge in $G_{SP}$. As a consequence $X_{\pi_{SP}(a)}\independent X_{\pi_{SP}(b)} \mid X_S$ where $S=\{\pi_{SP}(1),\dots , \pi_{SP}(\max(a,b)-1)\}\setminus\{\pi_{SP}(\min(a,b))\}$ and therefore  $c>\max(a,b)$. Hence $(j,k,\ell)$ is also a v-structure in $G_{SP}$. The case where $(j,k,\ell)$ is not a v-structure is analogous. 
\end{proof}

We now analyze consistency guarantees of the SP algorithm. For this we introduce the following assumption. %An issue of the SP algorithm is that the output $G_{SP}$ will often contain multiple DAGs, some of which may not belong to the same Markov equivalence class. In such a case, one approach is to output all DAGs and let the user decide which DAG is most appropriate. For the theoretical analysis and simulations in this paper, the presence of multiple Markov equivalence classes in $G_{SP}$ is interpreted as \emph{failure} of the SP algorithm, while \emph{sucess} denotes the event where $G_{SP}=\mathcal{M}(G^*)$.

\begin{definition}%[Sparsest Markov representation (SMR) assumption]
	\label{AssSMR}
	A pair $(G^*,\mathbb{P})$ satisfies the \emph{sparsest Markov representation} (\emph{SMR}) \emph{assumption} if $(G^*,\mathbb{P})$ satisfies the Markov property and $|G| > |G^*|$ for every DAG $G$ such that $(G, \mathbb{P})$ satisfies the Markov property and $G \notin \mathcal{M}(G^*)$.
\end{definition}
The SMR assumption asserts that the true DAG $G^*$ is the (unique up to Markov equivalence) sparsest DAG satisfying the Markov property. In the absence of additional information, the SMR assumption is a necessary condition for any algorithm that uses the CI relations of $\mathbb{P}$ to infer $G^*$, since if there is a DAG $G \notin \mathcal{M}(G^*)$ which satisfies the Markov property and is as sparse or sparser than $G^*$, there is no reason to select $G^*$ over $G$. Now we present our first main result showing that the SMR assumption is necessary and sufficient for consistency of the SP algorithm.

\begin{theorem}
	\label{ThmOracleMain}
	The SP algorithm outputs $G_{SP}\in\mathcal{M}(G^*)$ if and only if the pair $(G^*,\mathbb{P})$ satisfies the SMR assumption.
\end{theorem}

\begin{proof}
	First, let $(G^*,\mathbb{P})$ satisfy the SMR assumption and assume that $G_{SP}\notin\mathcal{M}(G^*)$. By Lemma \ref{LemSparsity}, $|G_{SP}| \leq |G^*|$,  which is a contradiction to the SMR assumption, since $G_{SP}$ satisfies the Markov property. Next, assume that $(G^*,\mathbb{P})$ does not satisfy the SMR assumption. Then there exists a DAG $G \notin \mathcal{M}(G^*)$ with $|G| \leq |G^*|$ that satisfies the Markov property. Without loss of generality, we choose $G$ with minimal number of edges. Let $\pi$ denote a permutation that is consistent with the partial order defined by the DAG $G$. Since $G_\pi$ satisfies the minimality assumption, $S(G_\pi) = S(G)$. If $|G| < |G^*|$, the SP algorithm will output $G\notin\mathcal{M}(G)$. If $|G| = |G^*|$, the SP algorithm will choose both $G$ and $G^*$. Hence the set $\{G_{SP}\}$ will include DAGs that are not in $\mathcal{M}(G^*)$.
\end{proof}

In the following result, we show that the SMR assumption is strictly weaker than the restricted-faithfulness assumption,  the weakest known sufficient condition for consistency of causal inference algorithms. 
\begin{theorem}
	\label{ThmMainPCSP}
	Let $(G^*,\mathbb{P})$ satisfy the Markov property. Then the restricted-faithfulness assumption implies the SMR assumption and there exist pairs $(G^*, \mathbb{P})$ that satisfy the SMR assumption but not the restricted-faithfulness assumption.
\end{theorem}
\begin{proof}
	We proved in Lemma~\ref{LemSparsity} (c) that restricted-faithfulness is sufficient for consistency of the SP algorithm. Applying Theorem~\ref{ThmOracleMain} we hence find that restricted-faithfulness implies the SMR assumption. We next construct an example of a $4$-node DAG that satisfies the SMR assumption but not restricted-faithfulness. Let $G^* = (\{1,2,3,4\}, E^*)$ be the 4-cycle with $E^* = \{(1,2),(1,4), (2,3), (3,4)\}$ and let $\mathbb{P}$ satisfy the CI relations $X_1\independent X_3 \mid X_2$,\; $X_2\independent X_4 \mid X_1, X_3$ and $X_1\independent X_2 \mid X_4$. The last CI relation violates adjacency-faithfulness. Such a violation is straightforward to construct: It corresponds to cancellation of the two paths between $X_1$ and $X_2$ conditioned on $X_4$, namely $X_1 \rightarrow X_4 \leftarrow X_3 \leftarrow X_2$ and $X_1 \rightarrow X_2$. On the other hand, using the SP algorithm it can be shown that any permutation other than the true permutation $(1,2,3,4)$ produces a DAG $G_\pi \notin \mathcal{M}(G^*)$ with more than $4$ edges. For example, consider the permutation $(1,4,2,3)$, i.e.,$\pi(1) = 1$, $\pi(2) = 4$, $\pi(3) =2$, $\pi(4) = 3$. In this case only the edge $(1,2)$ is omitted leading to a DAG with 5 edges. Similarly, one can check all other permutations, proving that $(G^*,\mathbb{P})$ satisfies the SMR assumption.
\end{proof}

We end this section by comparing the SMR assumption to two well-studied minimality assumptions, which we refer to as \emph{SGS-minimality}~\citep{spirtes:00} and \emph{P-minimality}~\citep{pearl00}. Both minimality assumptions encourage the selection of DAG models with fewer edges. %As we show in this section, they are weaker assumptions than the SMR assumption.
SGS-minimality was encountered already in Section~\ref{SecMain} and asserts that there exists no proper sub-DAG of $G^*$ that satisfies the Markov property with respect to $\mathbb{P}$. Since there exists an SGS-minimal DAG for every ordering, SGS-minimality is a significantly weaker condition than the SMR assumption. 

P-minimality with respect to a distribution $\mathbb{P}$ asserts that for DAGs satisfying the Markov property, models that entail more CI statements are preferred. Stated precisely, let $G$ and $G'$ be two DAGs that satisfy the Markov property with respect to $\mathbb{P}$ and let $D_{sep}(G)$ denote the d-separation statements for $G$ and $D_{sep}(G')$ the d-separation statements for $G'$. Then $G'$ is \emph{strictly preferred} to $G$ if $D_{sep}(G) \subsetneq D_{sep}(G')$. In other words, the DAG $G'$ entails a strict super-set of CI statements compared to the DAG $G$. The P-minimality assumption asserts that no DAG that satisfies the Markov property with respect to $\mathbb{P}$ is strictly preferred to the true DAG $G^*$. \citet{Zhang13} proved that P-minimality implies SGS-minimality. We now prove that the SMR assumption is strictly stronger than P-minimality.

\begin{theorem}
	\label{ThmPMin}
	\begin{enumerate}
		\item[(a)] If a DAG $G$ satisfies the SMR assumption with respect to a distribution $\mathbb{P}$, it also satisfies the P-minimality assumption with respect to $\mathbb{P}$.
		\item[(b)] There exist DAG models $(G^*, \mathbb{P})$, for which all DAGs that satisfy the SMR assumption belong to $\mathcal{M}(G^*)$ and there are DAGs $G \notin \mathcal{M}(G^*)$ that satisfy the P-minimality assumption.
	\end{enumerate}
\end{theorem}

\begin{proof}
	(a) \; Let $G$ be a DAG that satisfies the SMR assumption with respect to $\mathbb{P}$ and for contradiction assume $G$ does not satisfy P-minimality. Then there exists a DAG $G'$ that is Markov to $\mathbb{P}$ such that $D_{sep}(G) \subsetneq D_{sep}(G')$. Hence $G'\neq G$ and as a consequence of the SMR assumption, $|G'|\geq |G|$. We first consider the case where the skeleton of $G$ and $G'$ are identical.  So without loss of generality, there exists a v-structure $i\to k,$ $j\to k$ in $G'$ that is not in $G$. Hence there exists a subset $S$ with $k\in S$ that d-separates $i$ from $j$ in $G$ but does not d-separate $i$ from $j$ in $G'$. This is in contradiction with $D_{sep}(G) \subsetneq D_{sep}(G')$. As a consequence, the skeleta of $G$ and $G'$ are not identical. Since $|G'|\geq |G|$ this means that there exist vertices $i, j$ that are connected by an edge in $G'$ but not in $G$. Hence there exists a subset $S$ that d-separates $i$ from $j$ in $G$ but not in $G'$,  again contradicting $D_{sep}(G) \subsetneq D_{sep}(G')$. 
%	\vspace{0.2cm}
	
	(b) \; We consider the same example as in the proof of Theorem~\ref{ThmMainPCSP} (b). Let $G^* = (\{1,2,3,4\}, E^*)$ be the 4-cycle with $E^* = \{(1,2),(1,4), (2,3), (3,4)\}$ and let $\mathbb{P}$ satisfy the CI relations $X_1\independent X_3 \mid X_2$,\; $X_2\independent X_4 \mid X_1, X_3$ and $X_1\independent X_2 \mid X_4$. All DAGs that satisfy the SMR assumption belong to $\mathcal{M}(G^*)$. However, the DAG with edge set $\tilde{E} = \{(1,4),(1,3), (4,2), (4,3), (2,3)\}$ is Markov to $\mathbb{P}$, satisfies the P-minimality assumption, but is not in $\mathcal{M}(G^*)$.
\end{proof}

Theorem~\ref{ThmPMin} (b) proves that the converse of (a) is false and shows that while P-minimality tends to encourage sparser models, it is less powerful than the SMR assumption for determining the true Markov equivalence class of $G^*$.

\section{Effects of inferring CI relations from data}
\label{SecNoisy}

The CI relations used by the SP algorithm need to be estimated from data. To infer the CI relations, we apply the standard hypothesis testing framework as it was used also in earlier constraint-based methods (see, e.g.,~\citet{RobScheiSpirtWas,ZhangSpirtes03}). For brevity, we only present the formal testing framework when $\mathbb{P}$ is a Gaussian distribution, although the results in this section apply for any distribution in which suitable CI tests can be constructed. 

In the Gaussian setting, a CI relation $X_j \independent X_k \mid X_S$ is equivalent to zero partial correlation $\rho_{j,k | S} := \mbox{corr}(X_j, X_k\mid X_S)$. Hypothesis testing then boils down to
\begin{eqnarray*}
	\mathcal{H}_0 : \quad \rho_{j,k | S} & = & 0 \qquad\textrm{(i.e., $X_j \independent X_k \mid X_S$)},\\
	\mathcal{H}_1 : \quad \rho_{j,k | S} & \neq & 0 \qquad\textrm{(i.e., $X_j \notindependent X_k \mid X_S$)},
\end{eqnarray*}
for which a z-test based on Fisher's z-transform~\citep{Fisher15} can be built as follows: From the sample covariance matrix $\widehat{\Sigma}$ compute the sample correlation coefficients $\hat{\rho}_{j,k | S} = [\widehat{\Sigma} ]_{jk,jk} - [\widehat{\Sigma} ]_{jk,S} ([\widehat{\Sigma} ]_{S,S})^{-1} [\widehat{\Sigma} ]_{S, jk}$ via Schur complement. Next, compute Fisher's z-transform 
$$\hat{z}_{jk | S} = \frac{1}{2} \log \biggr (  \frac{1+\hat{\rho}_{j,k | S}}{1-\hat{\rho}_{j,k | S}}\biggr).$$ A test of size $\alpha$ is obtained using the test statistic $T_n = \sqrt{n - |S| - 3 |\hat{z}_{jk | S}|}$ with confidence interval $R_n  = (-\Phi^{-1}(1 -\alpha/2), \Phi^{-1}(1 -\alpha/2))$, where $\Phi$ denotes the cumulative distribution function of $\mathcal{N}(0,1)$~\citep{Fisher15}.

In the following, we provide guarantees for \emph{uniform consistency} of the SP algorithm, i.e.~when the conditional independence relations need to be inferred from data as $n\to\infty$. For a precise definition see~\citep[Theorem 2]{ZhangSpirtes03}. \citet{RobScheiSpirtWas} proved that the faithfulness assumption is not sufficient to guarantee uniform consistency of the SGS and PC algorithms. Even if the true distribution $\mathbb{P}$ is faithful to $G^*$, the empirical distribution $\mathbb{P}_n$ might not satisfy the faithfulness assumption and might lead to failure of the PC algorithm. To overcome this problem, \citet{ZhangSpirtes03} introduced the \emph{strong-faithfulness assumption} and proved that it ensures uniform consistency of the PC and SGS algorithms. To simplify notation, we define $\Omega_{\lambda}(\mathbb{P}) = \{(j,k,S)\; | \; |\textrm{corr}(j,k\mid S)| \leq \lambda \}$.% for $\lambda >0$. 

\begin{definition} \citep{ZhangSpirtes03}
	Given $\lambda\in (0,1)$, a Gaussian distribution $\mathbb{P}$ satisfies the \emph{$\lambda$-strong-faithfulness assumption} with respect to a DAG $G$ if for any $j,k\in [p]$ and $S\subset [p]\setminus\{j,k\}$,
	$$ j \textrm{ is d-separated from } k \textrm{ given }S  \quad\Longleftrightarrow \quad (j,k, S) \in \Omega_{\lambda}(\mathbb{P}). $$
\end{definition}

Along the lines of the strong-faithfulness assumption for constraint-based methods, we introduce the \emph{strong-SMR assumption}. % in order to guarantee uniform consistency of the SP algorithm. 
We define this assumption for Gaussian distributions, although it can be extended to other distributions by replacing partial correlation tests with mutual information tests. %Before stating the strong-SMR assumption, 
We point out that the faithfulness assumption, the Markov property, and the SMR assumption, which are defined for distributions $\mathbb{P}$, naturally extend to sets of CI relations. In particular, we will be discussing these assumptions with respect to $\Omega_{\lambda}(\mathbb{P})$. 

\begin{definition}
	\label{AssSMR}
	Given $\lambda >0$, a Gaussian distribution $\mathbb{P}$ satisfies the \emph{$\lambda$-strong sparsest Markov representation} (\emph{SMR}) \emph{assumption} with respect to a DAG $G^*$ if $G^*$ satisfies the Markov property with respect to the CI relations in $\Omega_{\lambda}(\mathbb{P})$ and $|G| > |G^*|$ for every DAG $G$ such that $G$ satisfies the Markov property with respect to $\Omega_{\lambda}(\mathbb{P})$ and $G \notin \mathcal{M}(G^*)$.
\end{definition}
The $\lambda$-strong SMR assumption asserts that $G^*$ satisfies the SMR assumption with respect to the CI relations in $\Omega_{\lambda}(\mathbb{P})$. Since strong-faithfulness asserts that $G^*$ is faithful to the CI relations in $\Omega_{\lambda}(\mathbb{P})$, it follows from Theorem \ref{ThmMainPCSP} that the $\lambda$-strong SMR assumption is strictly weaker than the $\lambda$-strong faithfulness assumption. 
%Definition~\ref{AssSMR} is stated only for Gaussian distributions. It could, however,
Using the framework developed by \citet{RobScheiSpirtWas} and used by \citet{ZhangSpirtes03}, we now prove uniform consistency of the SP algorithm.

\begin{theorem}
	\label{LemConsistency}
	Under the $\lambda$-strong SMR assumption there exists a sequence of hypothesis tests %$\phi_n : \chi^n \rightarrow \{0, 1\}$ 
	such that the SP algorithm is uniformly consistent.
\end{theorem}
%\noindent 
%\vspace{0.3cm}

\begin{proof}
	Let $\hat{\rho}_{jk\mid S}$ for $S\subset [p]\setminus\{j,k\}$ denote the sample partial correlation corresponding to $\rho_{jk\mid S}$. By Chebyshev's Inequality,
	$$\mathbb{P}\left( |\hat{\rho}_{jk\mid S} - \rho_{jk\mid S}|\geq \epsilon\right) \leq \frac{\var(\hat{\rho}_{jk\mid S})}{\epsilon^2}.$$
	The asymptotic distribution of the z-transform was computed by~\citet{Fisher15} to be
	$$\sqrt{n-|S|-3}\left(\hat{z}_{jk\mid S} - z_{jk\mid S}   \right) \xrightarrow{d}  \mathcal{N}(0,1).$$
	Since $\hat{\rho}_{jk\mid S} = \tanh(\hat{z}_{jk\mid S})$, then using the delta method we obtain the asymptotic distribution of $\hat{\rho}_{jk\mid S}$:
	$$\sqrt{n-|S|-3}\left( |\hat{\rho}_{jk\mid S} - \rho_{jk\mid S}\right) \xrightarrow{d} \mathcal{N}\left(0,\left(1-\rho_{jk\mid S}^2\right)^2\right).$$
	Hence, as a consequence, asymptotically, 
	$$\mathbb{P}\left( |\hat{\rho}_{jk\mid S} - \rho_{jk\mid S}|\geq \epsilon\right) \leq \frac{1}{(n-|S|-3)\,\epsilon^2}.$$
	More general distributional results on sample partial correlations can be found in \citep{Drton_Perlman}. So, for $n$ sufficiently large it holds with high probability that $|\hat{\rho}_{jk\mid S} - \rho_{jk\mid S}|\leq \epsilon$. Using the sample partial correlations $\hat{\rho}_{jk\mid S}$ for hypothesis testing, we can build a sequence of correlation hypothesis tests %$\phi^{\lambda}_n : \chi^n \rightarrow \{0, 1\}$ 
	such that the inferred CI relations equal the CI relations in $\Omega_{\lambda}(\mathbb{P})$ with probability approaching $1$. By applying the $\lambda$-strong SMR assumption, the SP algorithm recovers $\mathcal{M}(G^*)$ with probability approaching $1$. Hence the SP algorithm is uniformly consistent.
\end{proof}

We end this section by analyzing the consequences of type-I and type-II error rates of the CI tests for the SP algorithm. Note that a \mbox{type-I} error corresponds to missing a CI relation, whereas a type-II error corresponds to adding a CI relation. In order to contrast the sample setting from the oracle setting considered in~Section~\ref{SecMain}, we denote the estimated DAG by $\hat{G}$. The following theorem shows that if the type-II errors outweigh the type-I errors, %or more precisely, if there is a permutation $\pi$ such that $|S(\hat{G}_\pi)| \leq |S(G^*)|$, 
then the SP algorithm performs at least as well as the SGS algorithm in terms of recovering the skeleton of $G^*$.

\begin{theorem}
	\label{ThmErrors}
	If the SGS algorithm recovers the skeleton $S(G^*)$ based on the inferred CI relations and there exists a permutation $\pi$ such that $|S(\hat{G}_\pi)| \leq |S(G^*)|$, then $S(\hat{G}_{SP}) = S(G^*)$.
\end{theorem}
\begin{proof}
	Let $\hat{G}_{SGS}$ denote the output of the SGS algorithm. If the SGS algorithm recovers $S(G^*)$ then $S(\hat{G}_{SGS})=S(G^*)$. Since $S(\hat{G}_{SGS})\subset S(\hat{G}_{SP})$ (regardless of errors made when inferring the CI relations) and by assumption $|S(\hat{G}_\pi)| \leq |S(G^*)|$, it follows that $S(\hat{G}_{SP})=S(G^*)$.
\end{proof}
Consequently, if only type-II errors are made in inferring CI relations, then $|S(\hat{G}_{\pi})| \leq |S(G^*)|$ for some $\pi$ and the SP algorithm performs at least as well as the SGS algorithm in terms of recovering the skeleton. This is in general not the case in the presence of type-I errors, since it might happen that $|S(\hat{G}_\pi)| > |S(G^*)|$ for all $\pi$ while the SGS algorithm recovers $\mathcal{M}(G^*)$. The following $4$-node example illustrates this situation.

\begin{example}
	\label{ex_type1}
	Let $G^* = ([4], E^*)$ be a DAG with $E^* = \{(1, 2), (2, 3), (3, 4)\}$. Assume a type-I error occurs missing the CI relation $X_1 \independent X_4 \mid X_2, X_3$, while all other CI relations are inferred correctly. The SGS algorithm recovers the correct skeleton, since the CI relation $X_1 \independent X_4 \mid X_3$ leads to the removal of the edge $(1,4)$, although the CI relation $X_1 \independent X_4 \mid X_2, X_3$ is missing. On the other hand, the SP algorithm will include the edge $(1,4)$ for a permutation $\pi^*$ consistent with $G^*$ and result in $|\hat{G}_\pi| = 4 > |G^*|$. Furthermore, it can be shown that every permutation $\pi$ results in a DAG $\hat{G}_\pi$ such that $|\hat{G}_\pi| \geq 4$. So in this example, the SP algorithm fails to recover the correct skeleton. \hfill$\square$
\end{example}

Theorem~\ref{ThmErrors} and Example~\ref{ex_type1} show that the SGS algorithm is more robust to type-I errors in inferring the CI relations, while the SP algorithm is more robust to type-II errors, since by searching over all permutations it can exploit the global structure of $G^*$. This is also apparent in our simulations presented in~Section~\ref{SecSimulations}, where we analyze the frequency of errors made by adding or missing edges for various algorithms. %Our results indicate that as long as the size $\alpha$ of the hypothesis test is set sufficiently low, type-I errors arise very rarely and the SP algorithm will generally outperform both the SGS and PC algorithms. 

\section{SP algorithm for Gaussian DAG models}
\label{SecGaussian}

In this section, we analyze the SP algorithm in the case when $\mathbb{P}$ is a Gaussian distribution. In this case, $\mathbb{P}$ is defined by the following linear structural equations:
\begin{equation}
\label{EqnStruct}
X_k= \sum_{j < k}{A_{jk} X_j} + \epsilon_k,\qquad j = 1,2,...,p,
\end{equation}
where $\noiseVect = (\epsilon_1,\epsilon_2, \ldots ,\epsilon_p) \sim \mathcal{N}(0, D)$, $D = \mbox{diag}(\sigma_1^2, \sigma_2^2,...,\sigma_p^2)$ and $A_{jk} \neq 0$ if and only if $(j,k) \in E$. We assume that $\sigma_j^2 > 0$ for all $j$ to ensure that $\mathbb{P}$ has positive measure everywhere. Without loss of generality, we assume that the vertices of the underlying DAG $G$ are \emph{topologically ordered}, meaning that $j<k$ for all $(j,k)\in E$. Then the structural equations can be expressed in matrix form as follows:
\begin{eqnarray*} 
	\label{EqnMat}
	(I - A)^T \XVect = \noiseVect,
\end{eqnarray*}
where $\XVect = (X_1, X_2, ..., X_p)$ and $A \in \mathbb{R}^{p
	\times p}$ is a strictly upper triangular matrix. From $\noiseVect \sim \mathcal{N}(0,D)$, we obtain that
\begin{eqnarray*}
	\XVect \sim \mathcal{N}(0, [(I-A) D (I-A)^T]^{-1}).
\end{eqnarray*} 
To simplify notation we define $U:=I-A$ and the inverse covariance matrix $K:=UDU^T$. Note that $UDU^T$ is the upper Cholesky decomposition of $K$, which exists and is unique for any symmetric positive definite matrix. %The goal is then to determine the Markov equivalence class of a DAG $G$ based on the inverse covariance matrix $K$ which encodes all CI relations.
Let $K^{\pi}$ denote the inverse covariance matrix $K$ with the rows and columns permuted according to the permutation $\pi$ and let $K^{\pi} = U^{\pi} D^{\pi} U^{\pi T}$ denote the upper Cholesky decomposition of $K^{\pi}$. The following result by~\citet{Pourahmadi99} links the entries of $U^{\pi}$ to conditional covariances.

%we show that in the Gaussian setting finding the sparsest permutation $\pi$ is equivalent to finding the sparsest Cholesky decomposition of the inverse covariance matrix.

\begin{theorem}[\citet{Pourahmadi99}]
	\label{LemCholesky}
	For every permutation $\pi$ and $j < k$,
	\begin{equation}
	\label{EqnCovPi}
	U^{\pi}_{jk} = 0 \qquad\Longleftrightarrow\qquad \cov(X_{\pi(j)}, X_{\pi(k)}|X_S) = 0,
	\end{equation}
	where $S = \{\pi(1), \pi(2),...,\pi(k-1)\}\setminus\{\pi(j)\}$.
\end{theorem}
Applying this result, we can adapt the SP algorithm to the Gaussian setting and replace step (1) of the algorithm by:
\vspace{0.15cm}
\par
\begingroup
\leftskip4em
\rightskip\leftskip
	\emph{For each permutation $\pi$ form $K^\pi$ and compute the Cholesky decomposition $K^{\pi} = U^{\pi} D^{\pi} U^{\pi T}$. Let $E^{\pi} \subset \{(1,2), (1,3),\dots, (1,p), (2,3),\dots, (p-1,p)\}$, where  $(i,j) \in E^{\pi}$ if and only if $U^{\pi}_{ij} \neq 0$.}
	\par
\endgroup
\vspace{0.15cm}

Denoting the number of non-zero off-diagonal elements of $U^{\pi}$ by $\|U^{\pi} \|_0 := \sum_{i <j}{\mathbf{1}(U^{\pi}_{ij} \neq 0)}$, then in the oracle Gaussian setting the SP algorithm is equivalent to the following optimization problem:
\begin{center}
	\begin{tabular}{r l}
		$\arg \min_{\pi}$ & $\|U^{\pi}\|_0$ \\
		\mbox{subject to} & $K^{\pi} = U^{\pi} D^{\pi} U^{\pi T} \quad \textrm{(Cholesky decomposition)}.$
	\end{tabular} 
\end{center}
Hence in the noiseless Gaussian setting the SP algorithm boils down to finding the permutation of the rows and columns of the inverse covariance matrix that provides the sparsest Cholesky factorization. Although finding the sparsest Cholesky decomposition is NP-complete~\citep{Yan81}, a number of heuristic methods exist; see \citep{DavisBook} for an overview and \citep{GeorgeLiu89} for a review of a particular algorithm, the \emph{minimum degree algorithm}.

%\vspace{0.3cm}

This equivalence allows us to establish a connection to \emph{$\ell_0$-penalized maximum likelihood estimation} introduced by~\citet{geerpb12}. Given a sample covariance matrix $\widehat{\Sigma}_n = \frac{1}{n} \sum_{i=1}^{n} {X^{(i)} X^{(i) T}}$ based on $n$ i.i.d.~samples $X^{(i)}\sim\mathcal{N}(0, \Sigma^*)$, \citet{geerpb12} suggest learning $\mathcal{M}(G)$ by solving 
\begin{center}
	\begin{tabular}{r l}
		$\arg \min_{K, P, U, D}$ & $ \mbox{trace}(K \widehat{\Sigma}_n ) - \log \det(K) + C^2 \|U\|_0 $\\
		\mbox{subject to} & $K = P U D U^T P^T$,
	\end{tabular} 
\end{center}
where $P$ is a permutation matrix, $U$ an upper triangular matrix with all diagonal entries equal to $1$, $D$ a positive diagonal matrix, and $C^2$ a regularization parameter. They prove that if $C^2 = \mathcal{O}(\frac{\log p}{n})$ and under the so-called \emph{permutation beta-min condition}~\citep[page 9]{geerpb12} the penalized maximum likelihood estimator of $U$ corresponds to a DAG that is Markov with respect to $\mathcal{N}(0, \Sigma^*)$. 

In the oracle setting in which $\widehat{\Sigma}_n = \Sigma^*$ and $C \rightarrow \infty$, the $\ell_0$-penalized maximum likelihood approach reduces to
\begin{center}
	\begin{tabular}{r l}
		$\arg \min_{P, U, D}$ & $\|U\|_0$\\
		\mbox{subject to} & $(\Sigma^*)^{-1} = P U D U^T P^T,$
	\end{tabular} 
\end{center}
which is equivalent to the SP algorithm in the noiseless Gaussian setting. Hence, Theorem~\ref{ThmMainPCSP} implies that in the oracle setting $\ell_0$-penalized maximum likelihood estimation is consistent under strictly weaker conditions than faithfulness. This provides the first result that relates the permutation beta-min condition to the faithfulness assumption.

\section{Simulation comparison between SP and other algorithms}
\label{SecSimulations}

The stronger consistency guarantees obtained in this paper for the SP algorithm as compared to other causal inference algorithms suggest a higher recovery rate of $\mathcal{M}(G^*\!)$. In this section, we support our theoretical results with simulations on Gaussian DAG models  in the sample setting where the CI relations are inferred from data. In addition to the constraint-based PC and SGS algorithms, we compare the SP algorithm to the popular score-based GES and hybrid MMHC algorithms. For the PC, SGS, and GES algorithms we used the R package 'pcalg'~\citep{kaetal10}, while for the MMHC algorithm, we used the R package 'bnlearn'~\citep{Scutari10}. This is just a selection of algorithms and due to computational limits of our algorithm we only performed simulations on small DAGs of up to 8 nodes. The main purpose of this simulation study is to analyze empirically how much weaker the strong SMR assumption is as compared to the assumptions needed for consistency of other algorithms.

The simulation study was conducted as follows: $100$ realizations of a $p$-node random Gaussian DAG model were constructed with $p \in \{3,4,5,6,7,8\}$ for different expected neighborhood sizes (i.e., edge probabilities) with edge weights $A_{jk}$ chosen uniformly in $[-1,-0.25] \cup [0.25, 1]$, ensuring the edge weights are bounded away from $0$. The considered expected neighborhood sizes range from 0.2 (very sparse DAG) to $p-1$ (complete DAG). Subsequently, $n$ samples were drawn from the distribution induced by the Gaussian DAG model. We report the results for $n\in\{1000, 10000\}$ and $p \in \{5, 8\}$. The CI relations were estimated using Fisher's z-transform as outlined in Section~\ref{SecNoisy}. The size of this hypothesis test $\alpha \in \{ 0.01, 0.001, 0.0001 \}$ was selected since empirically these values led to a good performance of the PC and SGS algorithms on the considered examples. This is consistent with the findings in the simulation results by \citet{Kalisch07}. For GES and MMHC we used the BIC as scoring function.

Figures~\ref{fig_simulations_p5} and~\ref{fig_simulations_p8} display the proportion of simulations out of $100$ in which the algorithms recovered the skeleton $S(G^*)$ for different expected neighborhood sizes. An issue of the SP algorithm is that the output $\{G_{SP}\}$ will often contain multiple DAGs, some of which may not belong to the same Markov equivalence class. In practice, one would then output all DAGs and let the user decide which DAG is most appropriate. Here, however, we recorded this as a failure of the SP algorithm, making the comparison less favorable for the SP algorithm. Note that since adjacency faithfulness is a necessary and sufficient condition for recovering $S(G^*)$ for the SGS algorithm and the same holds for the SMR assumption for the SP algorithm, these figures also provide a direct comparison between the adjacency faithfulness and SMR assumptions. Figures~\ref{fig_add_edge_p8} and~\ref{fig_miss_edge_p8} show the proportion of simulations out of $100$ in which the algorithms outputted a skeleton with additional and missing edges respectively for 8-node random DAG models. Figures~\ref{fig_add_edge_p8} and~\ref{fig_miss_edge_p8} allow us to analyze the strengths and weaknesses of each algorithm.

The results in Figures~\ref{fig_simulations_p5} and~\ref{fig_simulations_p8} show that the SP algorithm recovers the true skeleton $S(G^*)$ more frequently than the constraint-based SGS and PC algorithms. In addition, Figures~\ref{fig_add_edge_p8} and~\ref{fig_miss_edge_p8} show that this is mainly due to the fact that the SGS and PC algorithms miss edges in the skeleton, thereby supporting our theoretical findings. As expected, the SGS and PC algorithms are similar in terms of performance, but the PC algorithm tends to include additional edges more often than the SGS algorithm (see Figure~\ref{fig_add_edge_p8}), while the SGS algorithm tends to miss edges more often than the PC algorithm (see Figure~\ref{fig_miss_edge_p8}). The SP algorithm also outperforms the GES and MMHC algorithms as long as the DAG is not fully connected (i.e.~neighborhood size of $p-1$). The increased performance of GES for fully connected DAGs can be explained by the tendency of GES to include many additional edges as seen in Figure~\ref{fig_add_edge_p8}. We only show the results for $p=5$ and $p=8$, but it is clear from our simulation results on $p\in\{3,4,5,6,7,8\}$ nodes that the performance of the SP algorithm increases compared to the performance of the PC, SGS, GES and MMHC algorithms for increasing number of nodes. 
%This increase in accuracy, however, comes at a computational cost, which makes it difficult to apply the SP algorithm on DAGs with more than 10 nodes. While the focus of this work was on consistency guarantees, a greedy version of the SP algorithm that greedily optimizes over the space of permutations, is scalable to hundreds of nodes, and also comes with high-dimensional consistency guarantees has recently been developed~\citet{greedy_SP}.

%Since the focus of this work is on consistency guarantees and not on computation or scalability, a thorough computational comparison has not been conducted. It must be pointed out, however, that using our current implementation of the SP algorithm which exhaustively searches over all permutations, running the SP algorithm requires significantly greater computational resources compared to the other four algorithms. Speeding up the SP algorithm by developing efficient searches through the space of all permutations or using heuristics for finding sparse Cholesky decompositions remains an open and important research direction.

\section{Discussion and future work}

In this paper, we showed that it is possible to learn the underlying Markov equivalence class under strictly weaker conditions than restricted-faithfulness, the weakest known sufficient condition for consistency of causal inference algorithms for learning $\mathcal{M}(G^*)$. The proposed SP algorithm is a hybrid method that chooses the permutation which yields the DAG with the fewest number of edges. We proved that the SP algorithm is consistent under the SMR assumption, which is strictly weaker than the restricted-faithfulness assumption and we also compared these assumptions via simulations. When the variables follow a Gaussian distribution, we showed that in the oracle setting the SP algorithm is equivalent to finding the sparsest Cholesky factorization of the inverse covariance matrix and to $\ell_0$-penalized maximum likelihood estimation.

The stronger consistency guarantees of the SP algorithm, however, come at a large computational cost, which makes it difficult to apply the SP algorithm on DAGs with more than 10 nodes. Our results indicate the presence of a statistical-computational trade-off for causal inference algorithms; see also~\citep{Spirtes_3faces}. While the focus of this work was on consistency guarantees, a greedy version of the SP algorithm that greedily optimizes over the space of permutations, is scalable to hundreds of nodes, and also comes with high-dimensional consistency guarantees has recently been developed~\citep{greedy_SP}. In line with the presence of a statistical-computational trade-off, the assumptions required for consistency of the greedy SP algorithm are stronger than the SMR assumption, but still weaker than restricted-faithfulness~\citep{greedy_SP}. 

It is an important area for future research to better understand this statistical-computational trade-off. Given unlimited computation time, we conjecture that the SP algorithm reaches the information-theoretic limit, i.e., that the SMR assumption is the weakest assumption that guarantees consistency of any algorithm for learning $\mathcal{M}(G^*)$. More generally, it would be interesting to know whether any of the available causal inference algorithms are optimal in terms of statistical-computational trade-off. A better understanding of this trade-off is crucial in order to choose the optimal method for different applications with a fixed computational budget or sample size.

While in this paper we did not make any distributional assumptions, it was recently shown that such assumptions may allow the identification of the underlying causal DAG (and not just its Markov equivalence class) and this without requiring the faithfulness assumption. This was shown for non-Gaussian linear models~\citep{Shimizu_2006} and models with additive noise~\citep{Peters_2011,Peters_2014,Hoyer_2009}. These are different kinds of statistical assumptions and it would be interesting to understand the statistical-computational trade-off also in this setting.

%Although the SP algorithm requires checking all $p!$ permutations, we believe that the connection to sparse Cholesky factorization, for which there are good heuristics, might make our approach computationally feasible also for large graphs. It would then be interesting to define a computationally efficient version of the SP algorithm also for the high-dimensional setting with imposed sparsity. One would like to understand under which conditions the $\lambda$-strong SMR assumption is also necessary for uniform consistency and describe how the parameter $\lambda$ depends on $n$ and the graph structure.

%We proved uniform consistency of the SP algorithm under the $\lambda$-strong SMR assumption, which is strictly weaker than the $\lambda$-strong faithfulness assumption introduced in~\citep{ZhangSpirtes03}. It would be interesting to study how the strong SMR assumption compares to the $\beta$-min condition in~\citep{geerpb12}, which is sufficient for the $\ell_0$-penalized maximum likelihood approach, and also how it compares to consistency conditions for the GES and MMHC algorithms. One approach is to find a geometric description of the strong SMR assumption to bound the proportion of distributions that satisfy this assumption, similarly as for the strong faithfulness assumption in~\citep{UhlerRasBuYu}. However, this would be more difficult than in~\citep{UhlerRasBuYu}, since it would have to involve also a combinatorial argument. 

%\ack{I would like to thank....}

\begin{figure}[h!]
	\centering
	\subfigure[$p\!=\!5$, $n\!=\!10^3$, $\alpha\!=\!10^{-2}$]{\includegraphics[scale=0.22]{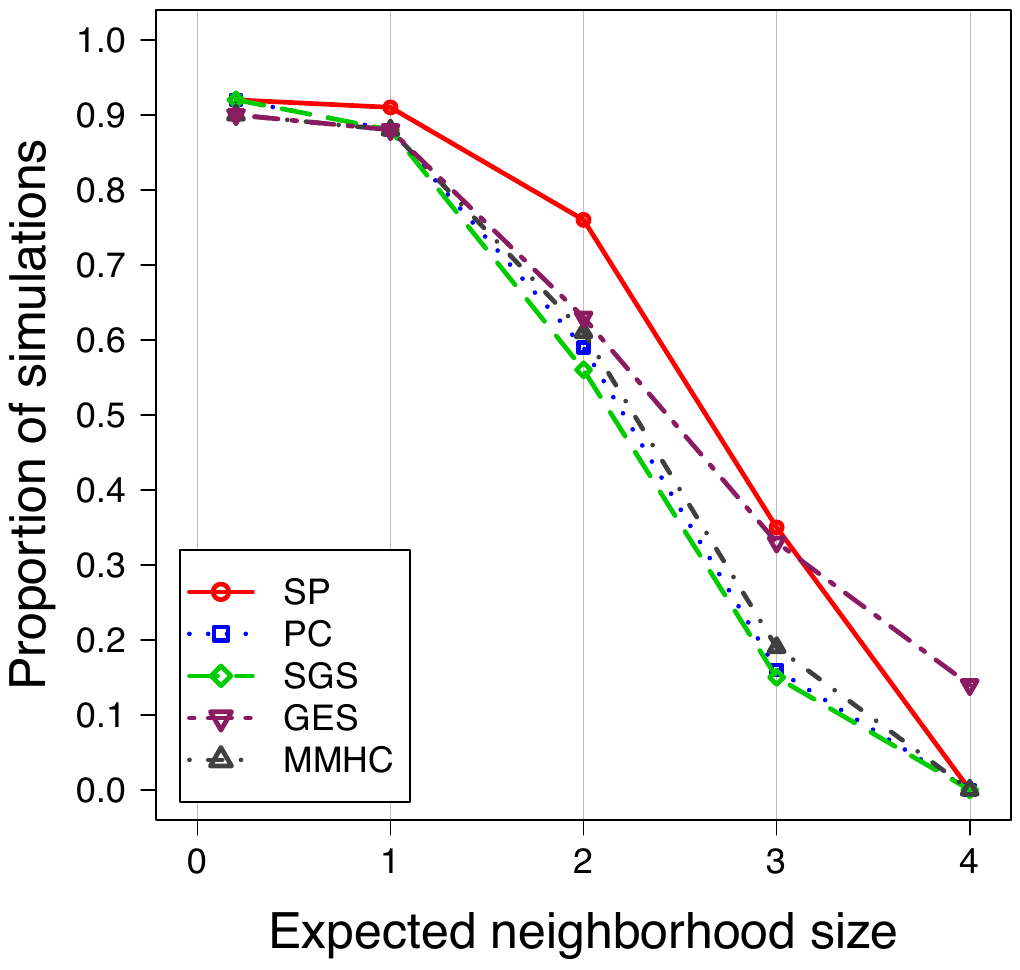}\label{fig:graph_1}}\;\,\,
	\subfigure[$p\!=\!5$, $n\!=\!10^3$, $\alpha\!=\!10^{-3}$]{\includegraphics[scale=0.22]{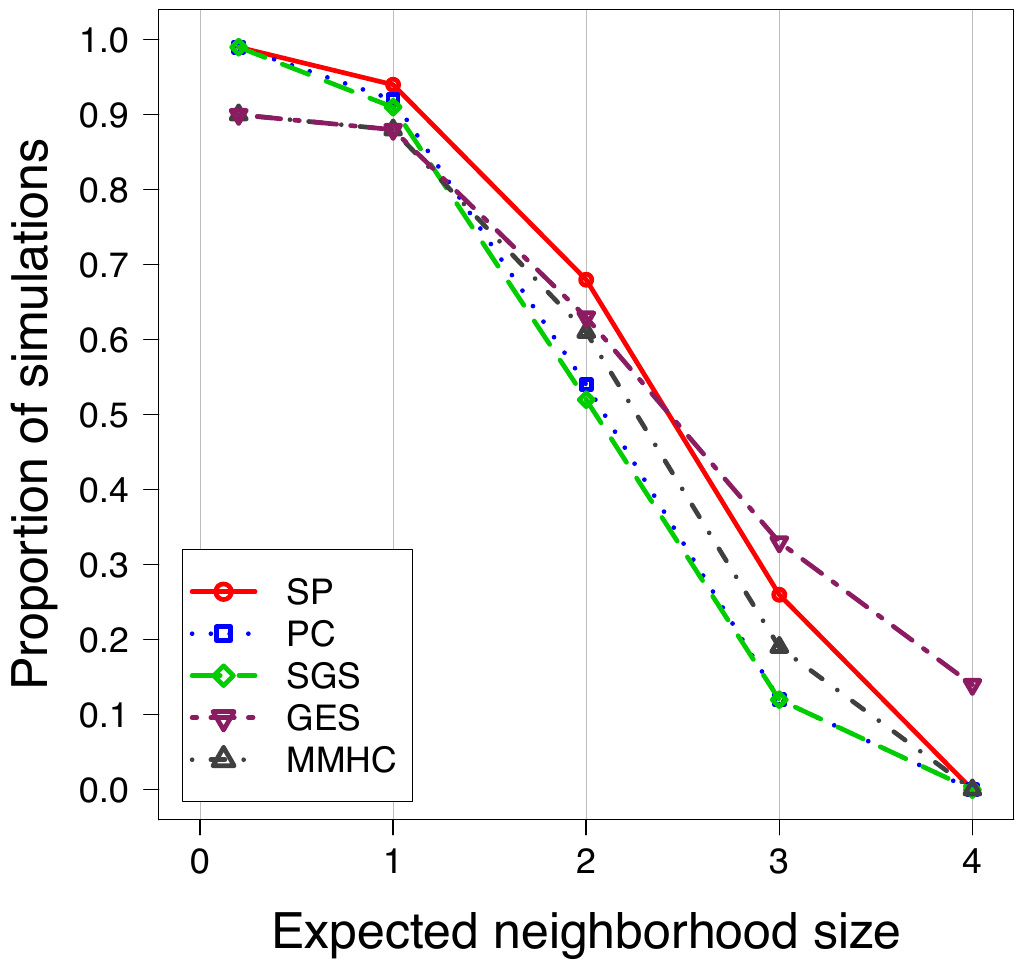}\label{fig:graph_2}}\;\,\,
	\subfigure[$p\!=\!5$, $n\!=\!10^3$, $\alpha\!=\!10^{-4}$]{\includegraphics[scale=0.22]{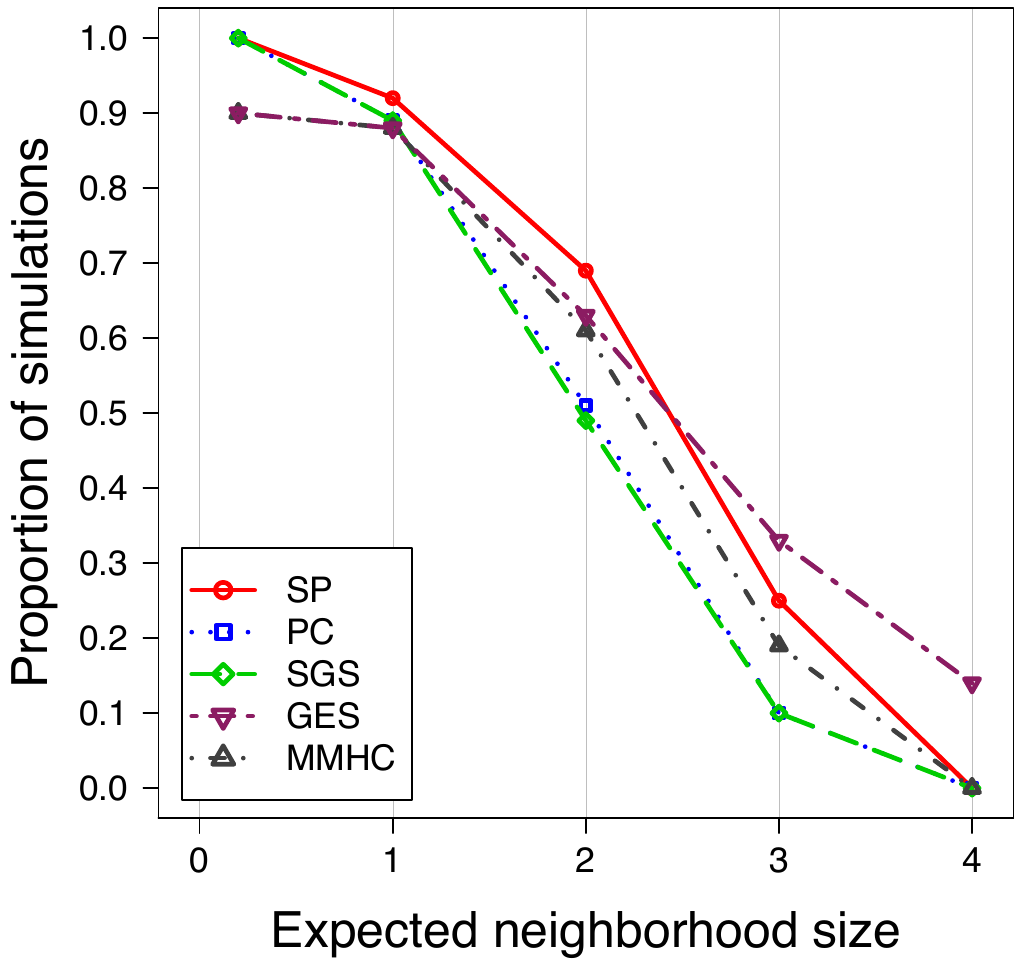}\label{fig:graph_3}}\\
	\subfigure[$p\!=\!5$, $n\!=\!10^4$, $\alpha\!=\!10^{-2}$]{\includegraphics[scale=0.22]{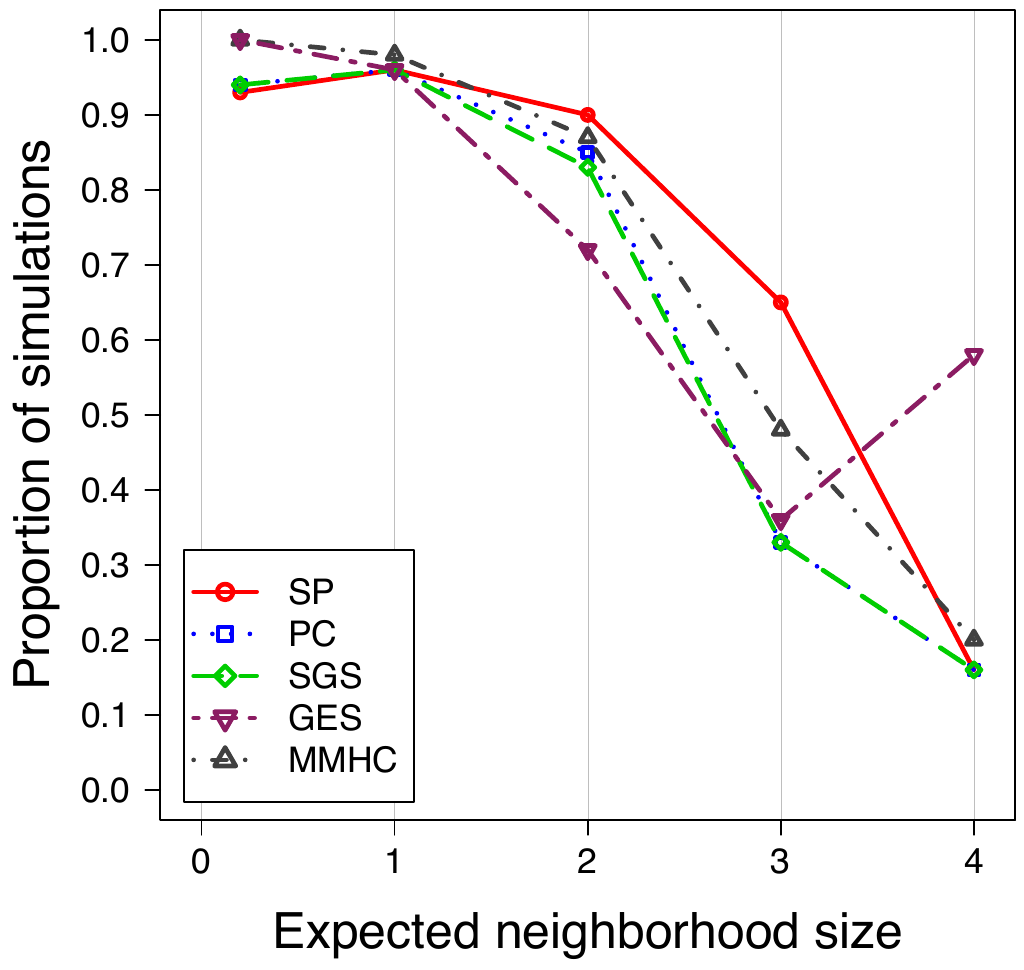}\label{fig:graph_4}}\;\,\,
	\subfigure[$p\!=\!5$, $n\!=\!10^4$, $\alpha\!=\!10^{-3}$]{\includegraphics[scale=0.22]{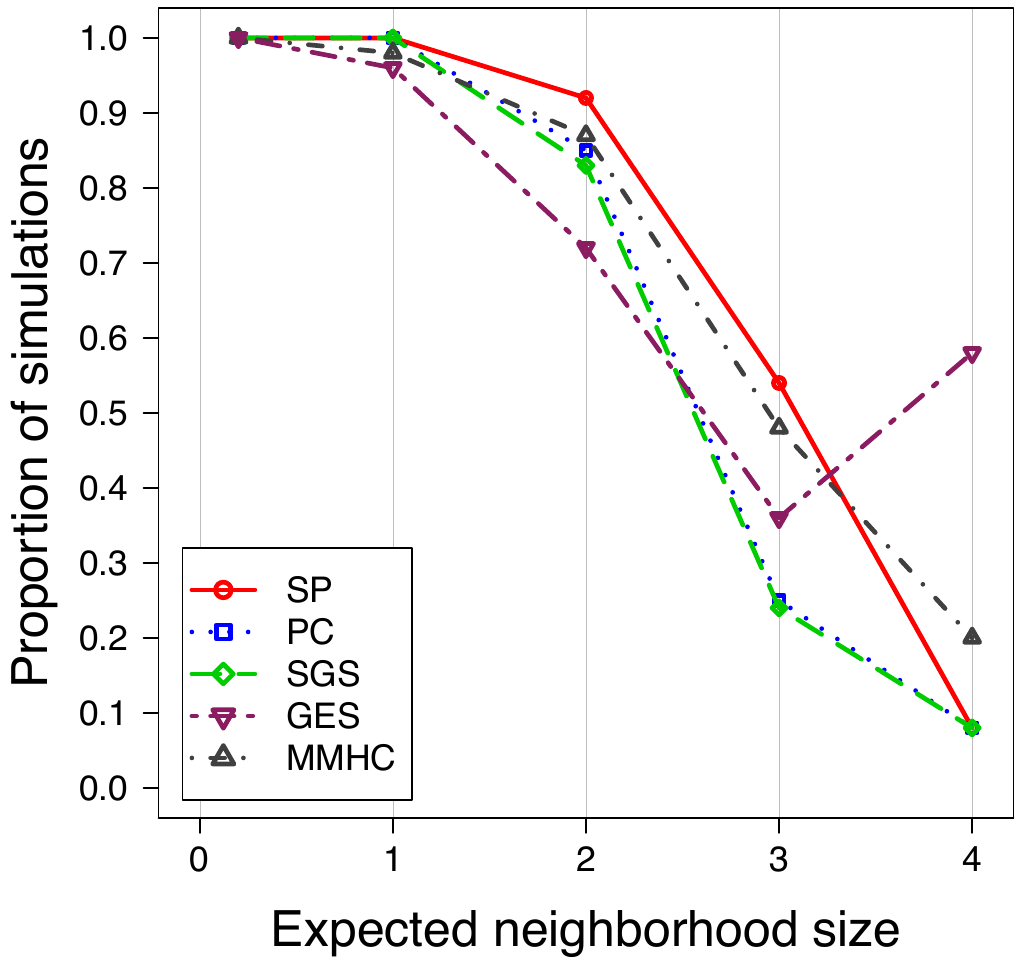}\label{fig:graph_5}}\;\,\,
	\subfigure[$p\!=\!5$, $n\!=\!10^4$, $\alpha\!=\!10^{-4}$]{\includegraphics[scale=0.22]{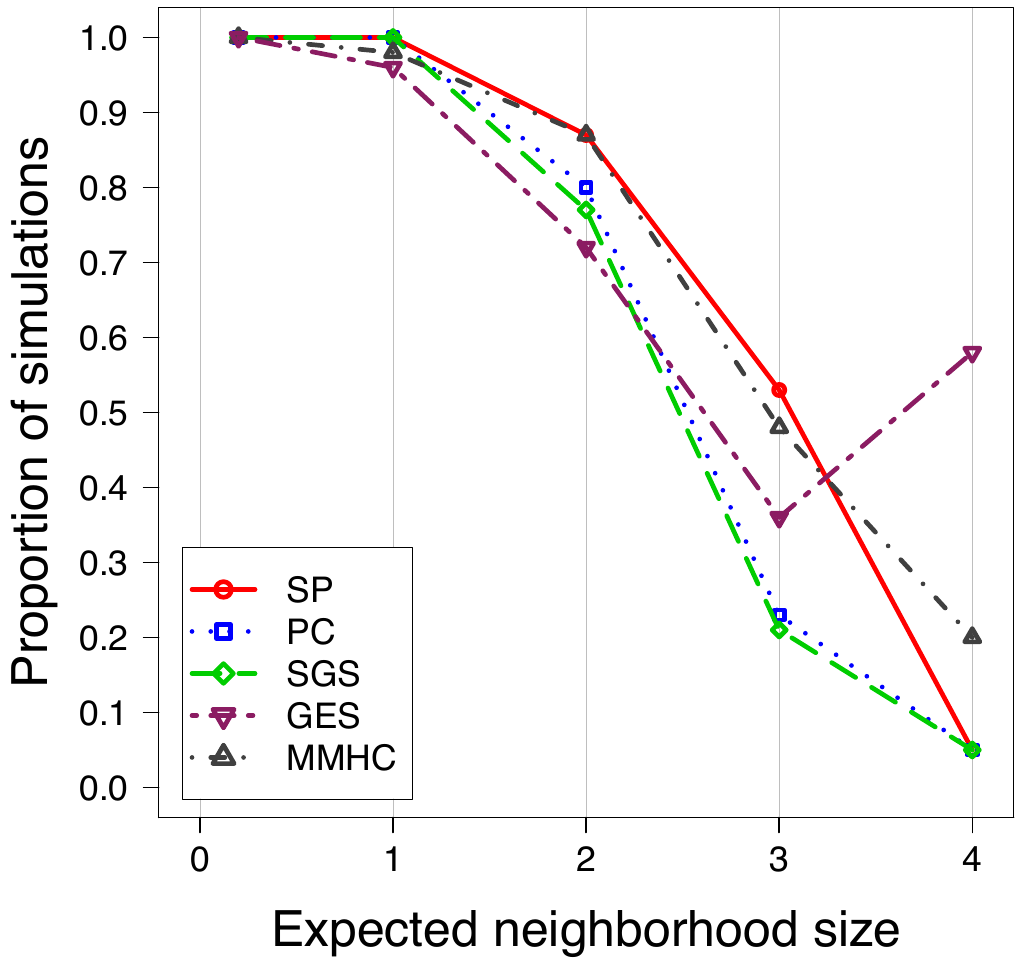}\label{fig:graph_6}}\\
	\caption{Proportion of $100$ simulations in which the algorithms recovered the skeleton $S(G^*)$ for 5-node random DAG models.}
	\label{fig_simulations_p5}
\end{figure}
\begin{figure}[h!]
	\centering
	\subfigure[$p\!=\!8$, $n\!=\!10^3$, $\alpha\!=\!10^{-2}$]{\includegraphics[scale=0.22]{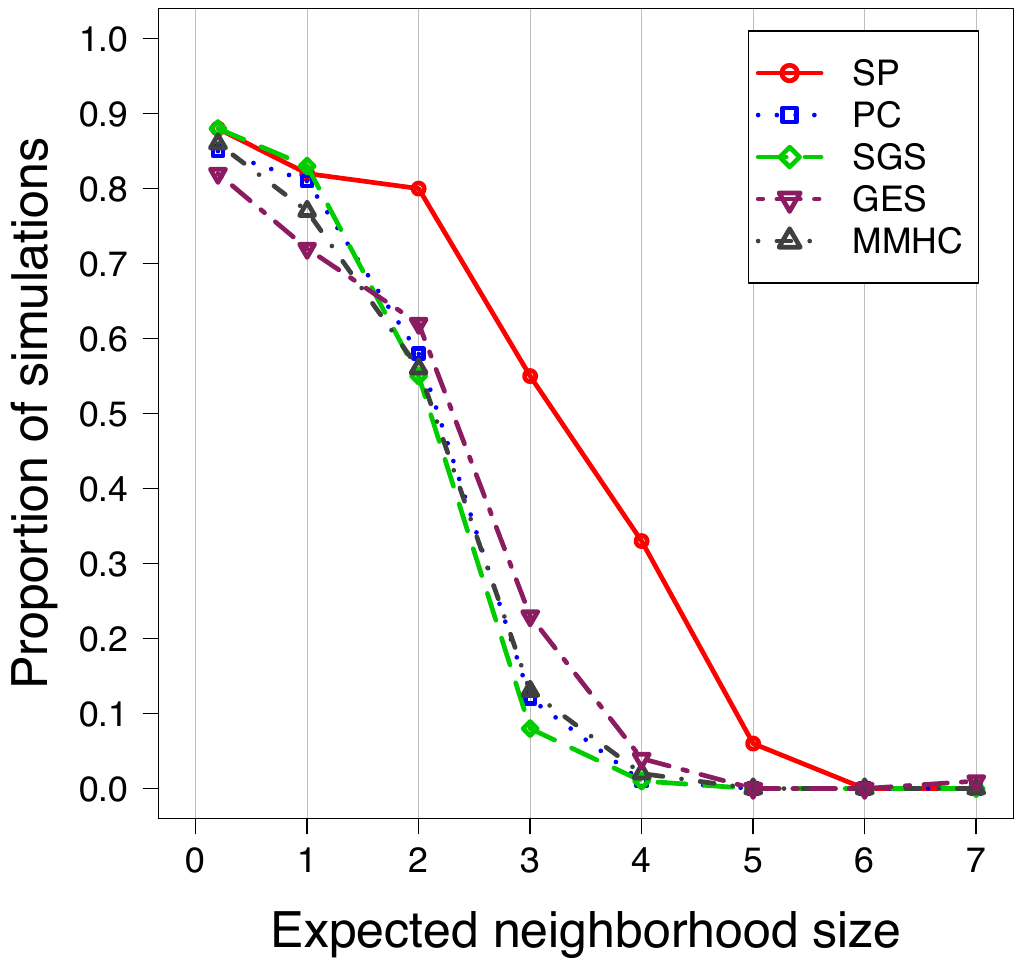}\label{fig:graph_7}}\;\;
	\subfigure[$p\!=\!8$, $n\!=\!10^3$, $\alpha\!=\!10^{-3}$]{\includegraphics[scale=0.22]{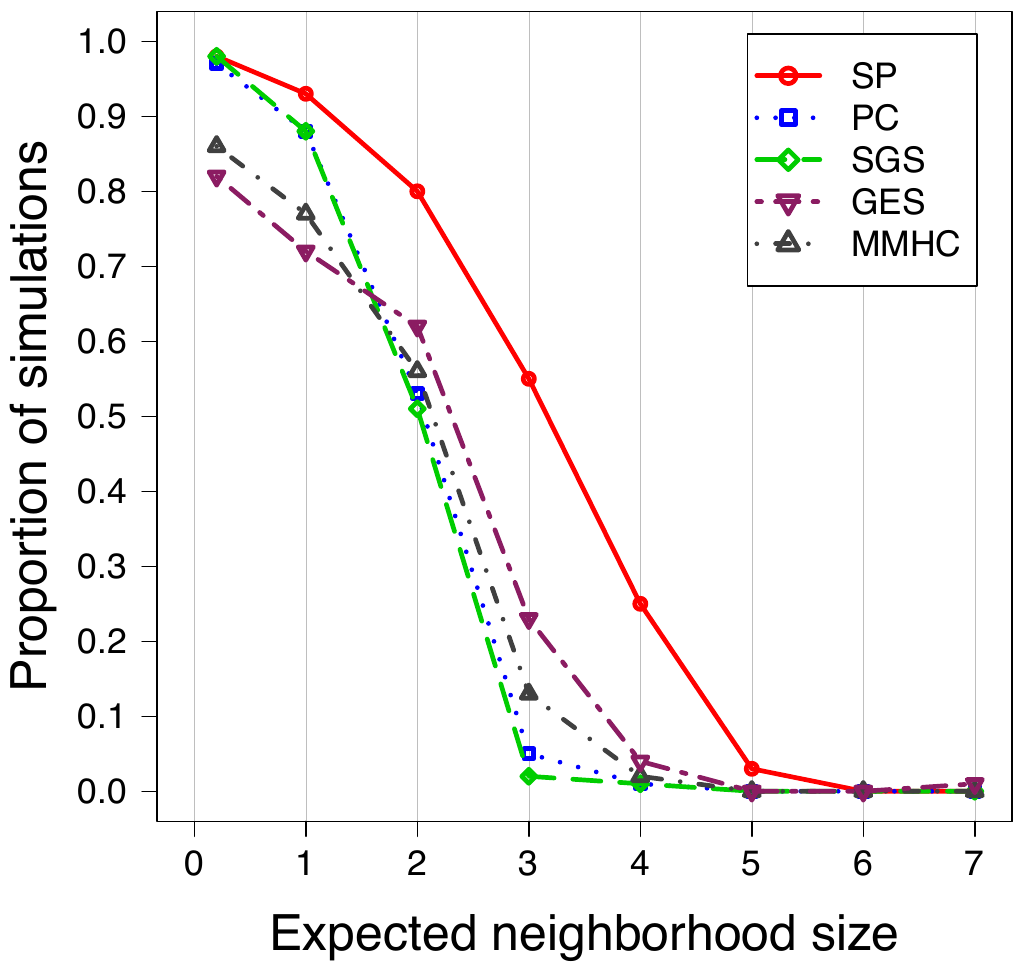}\label{fig:graph_8}}\;\;
	\subfigure[$p\!=\!8$, $n\!=\!10^3$, $\alpha\!=\!10^{-4}$]{\includegraphics[scale=0.22]{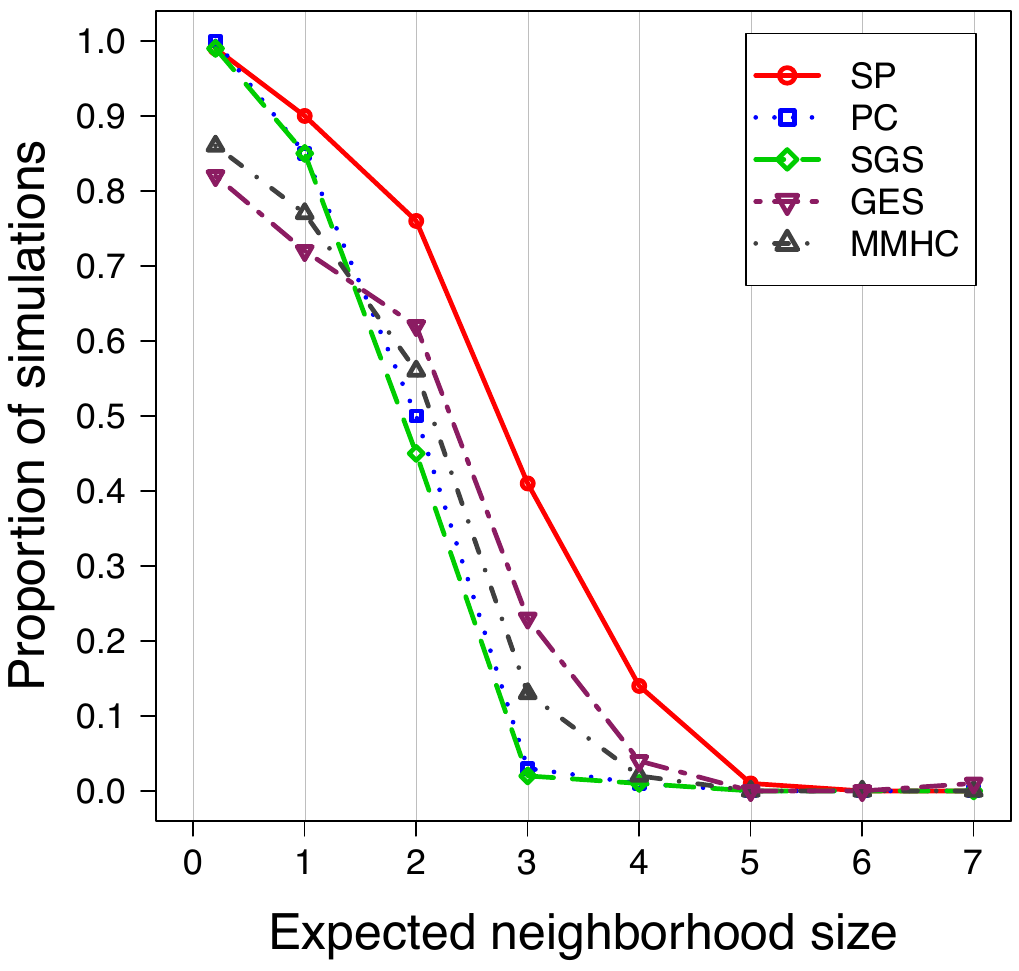}\label{fig:graph_9}}\\
	\subfigure[$p\!=\!8$, $n\!=\!10^4$, $\alpha\!=\!10^{-2}$]{\includegraphics[scale=0.22]{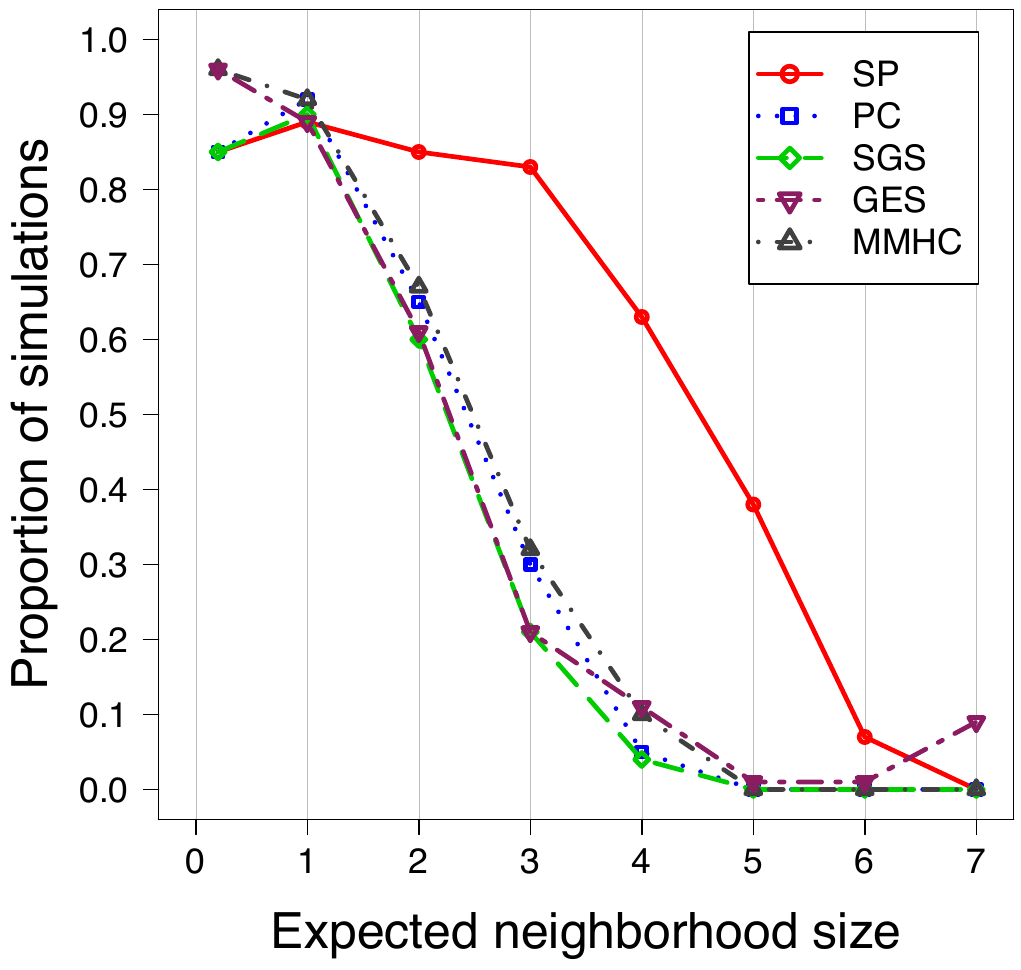}\label{fig:graph_10}}\;\;
	\subfigure[$p\!=\!8$, $n\!=\!10^4$, $\alpha\!=\!10^{-3}$]{\includegraphics[scale=0.22]{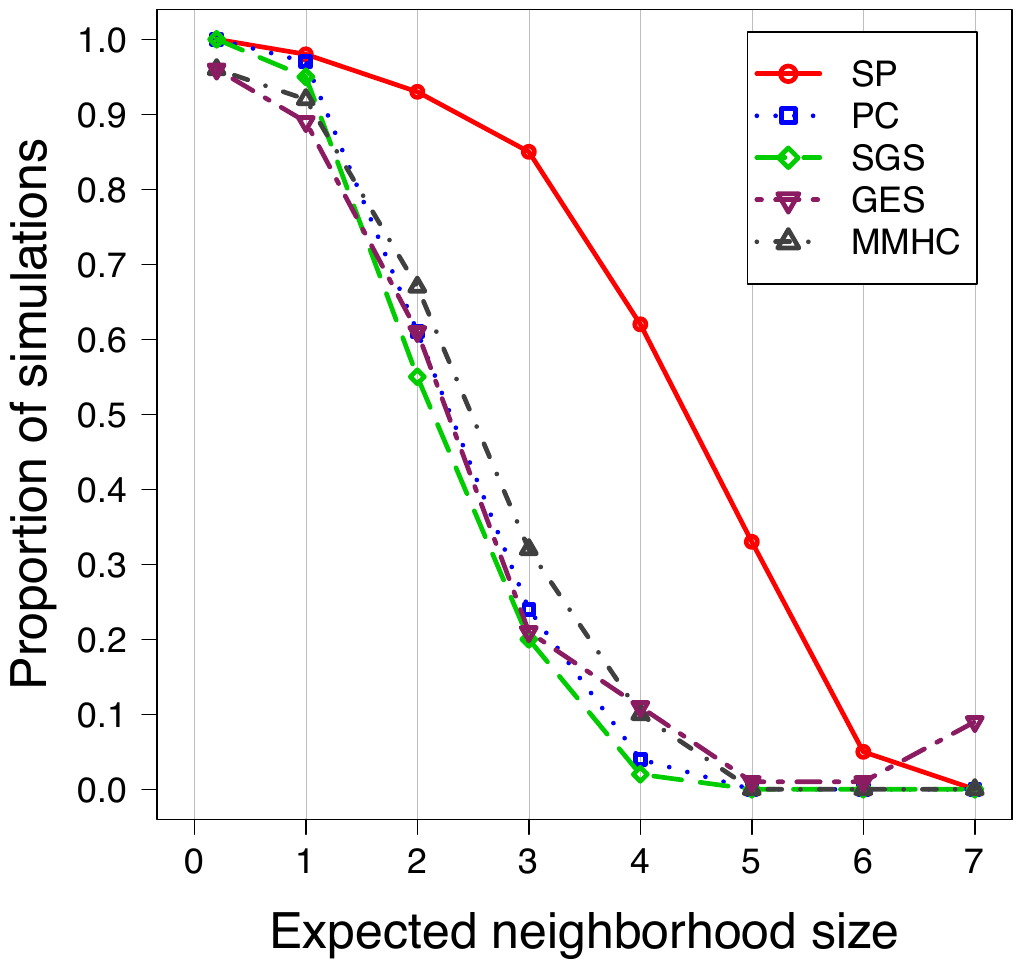}\label{fig:graph_11}}\;\;
	\subfigure[$p\!=\!8$, $n\!=\!10^4$, $\alpha\!=\!10^{-4}$]{\includegraphics[scale=0.22]{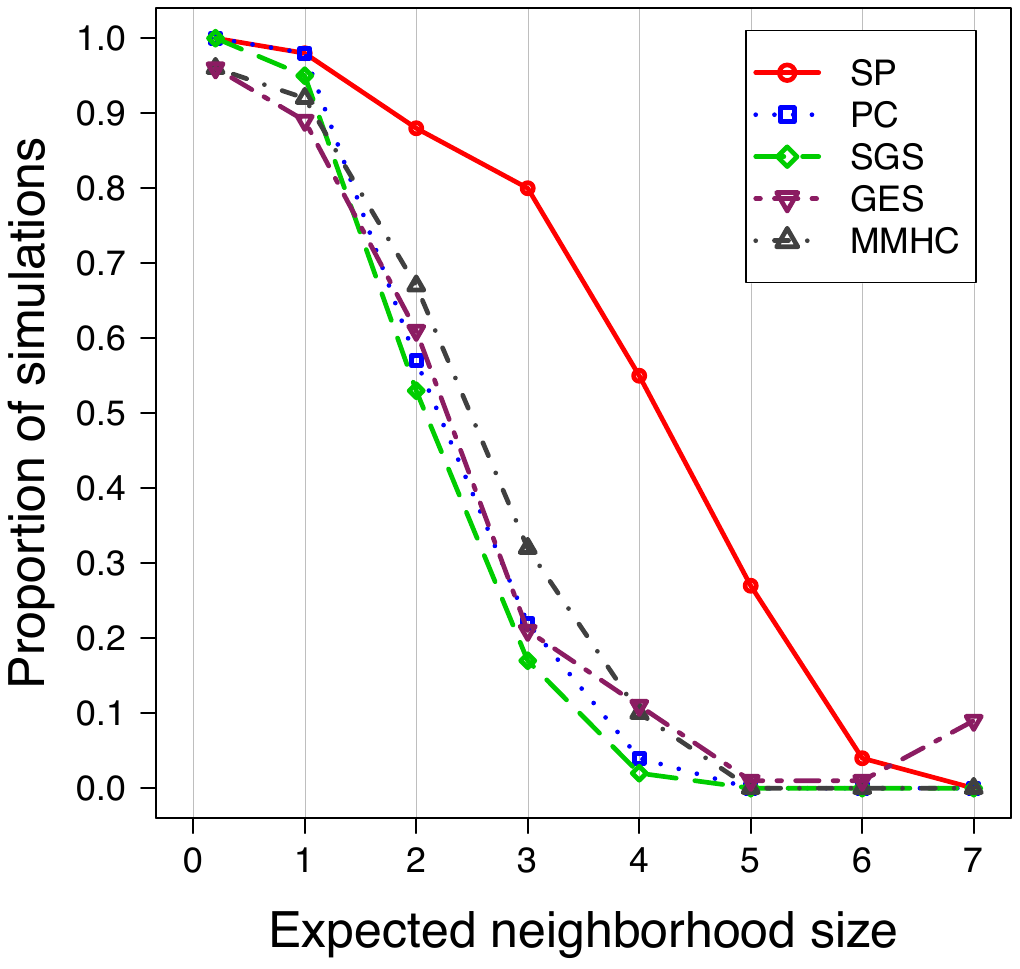}\label{fig:graph_12}}\\
	\caption{Proportion  of $100$ simulations in which the algorithms recovered the skeleton $S(G^*)$ for 8-node random DAG models.}
	\label{fig_simulations_p8}
\end{figure}

\begin{figure}[h!]
	\centering
	\subfigure[$p=8$, $n=10^4$, $\alpha=10^{-2}$]{\includegraphics[scale=0.22]{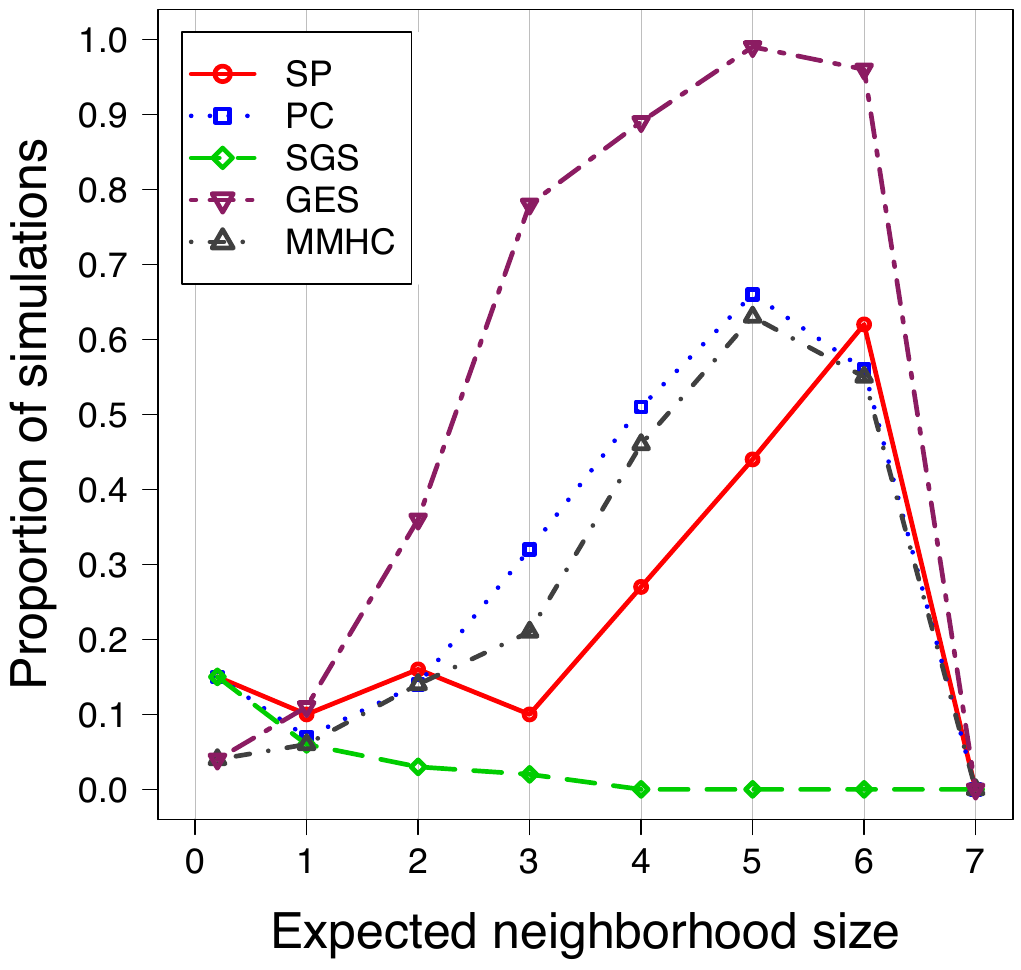}\label{fig:graph_13}}\;\;
	\subfigure[$p=8$, $n=10^4$, $\alpha=10^{-3}$]{\includegraphics[scale=0.22]{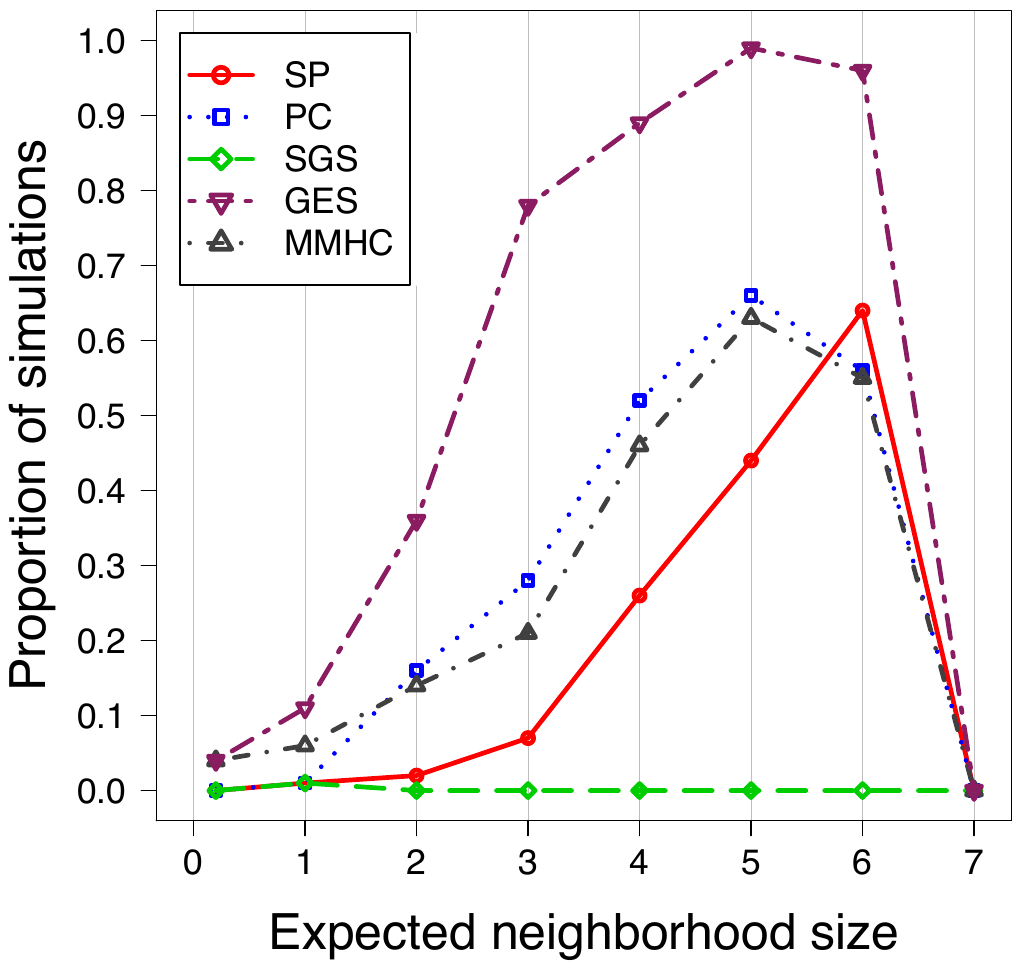}\label{fig:graph_14}}\;\;
	\subfigure[$p=8$, $n=10^4$, $\alpha=10^{-4}$]{\includegraphics[scale=0.22]{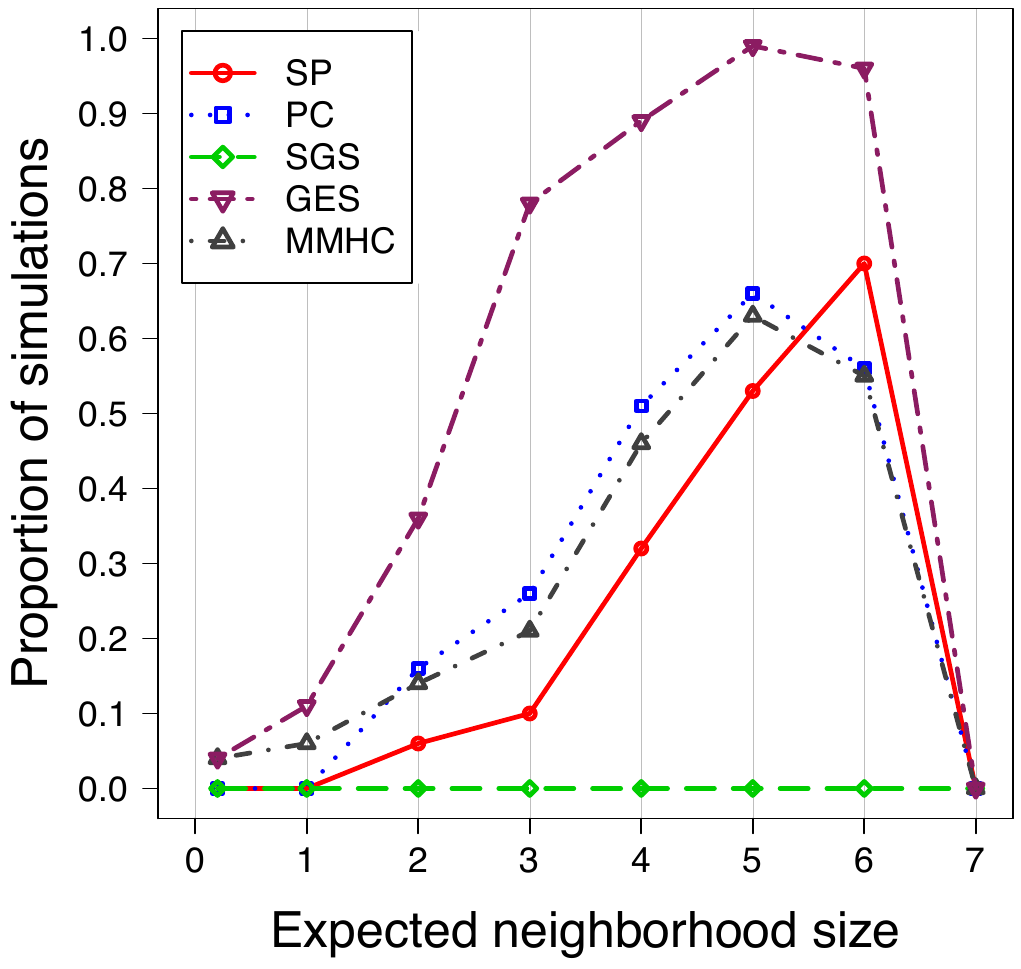}\label{fig:graph_15}}
	\caption{Proportion of simulations out of $100$ simulations in which the algorithms outputted a skeleton with additional edges for 8-node random DAG models.}
	\label{fig_add_edge_p8}
\end{figure}

\begin{figure}[h!]
	\centering
	\subfigure[$p=8$, $n=10^4$, $\alpha=10^{-2}$]{\includegraphics[scale=0.22]{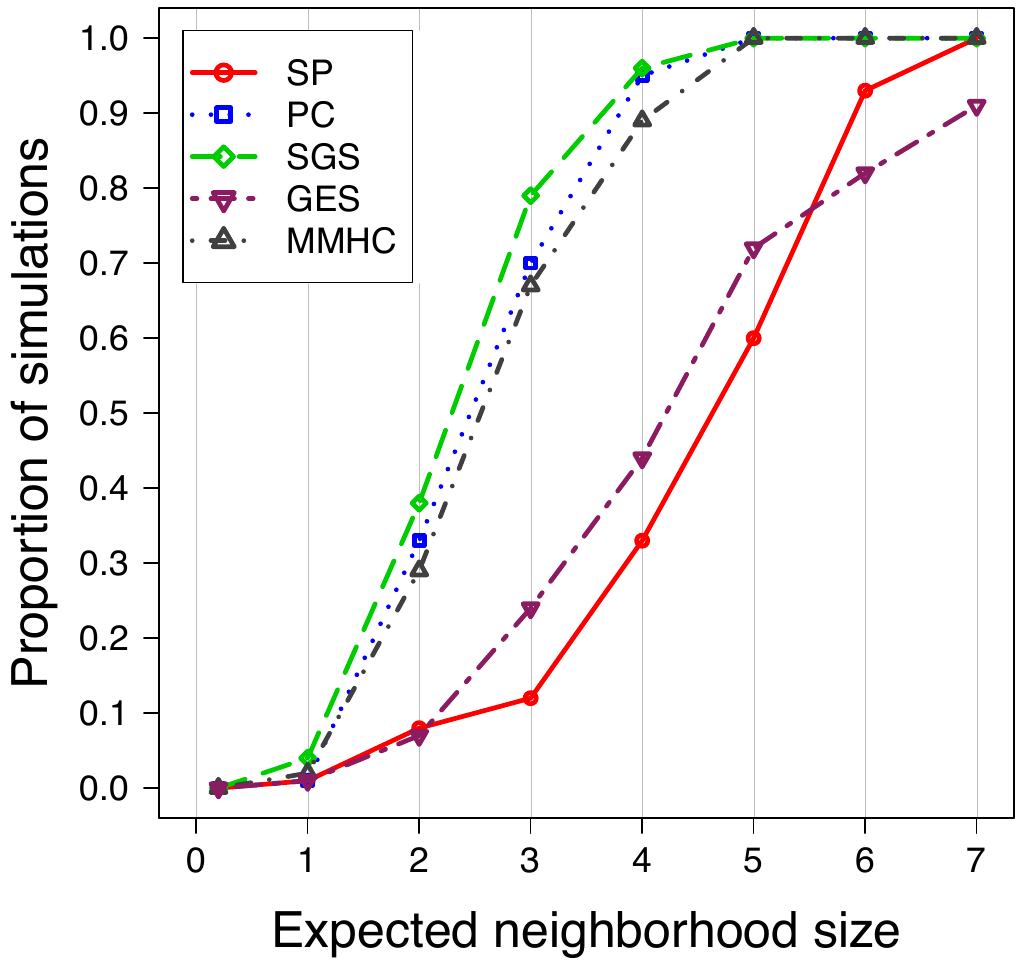}\label{fig:graph_16}}\;\;
	\subfigure[$p=8$, $n=10^4$, $\alpha=10^{-3}$]{\includegraphics[scale=0.22]{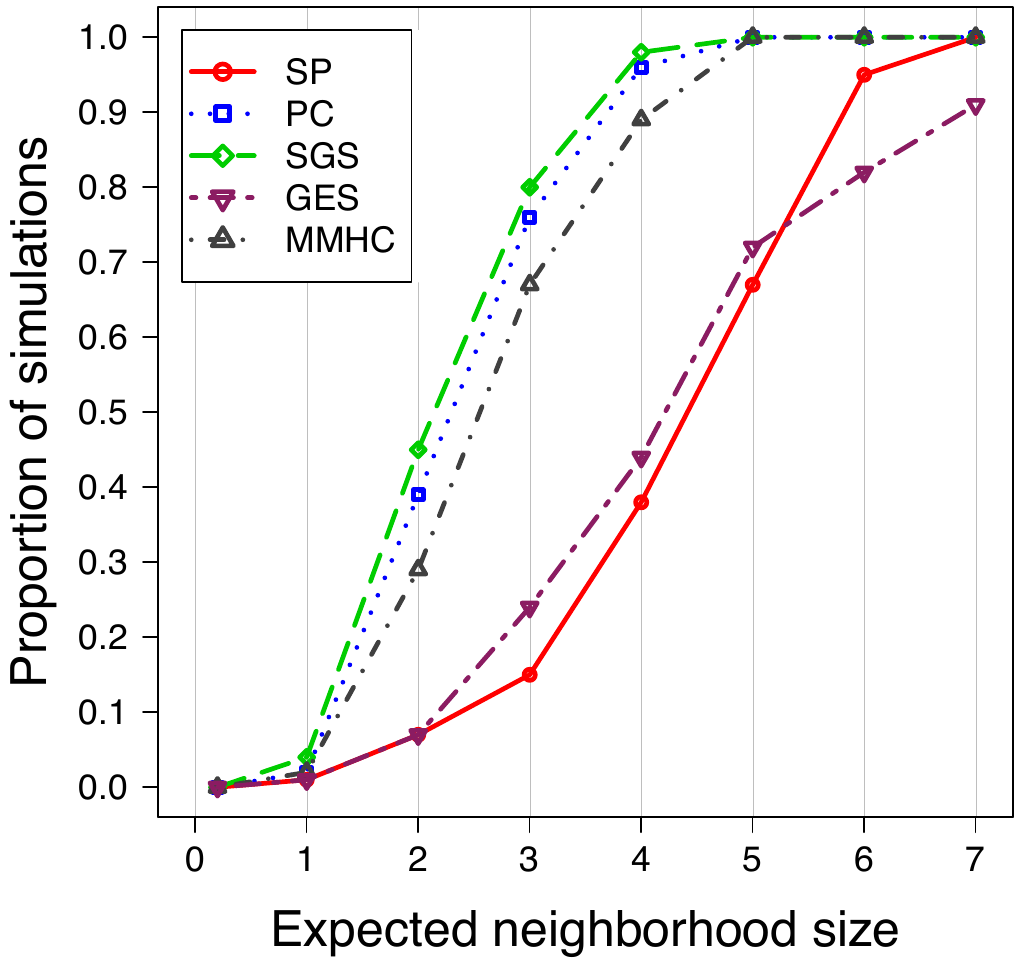}\label{fig:graph_17}}\;\;
	\subfigure[$p=8$, $n=10^4$, $\alpha=10^{-4}$]{\includegraphics[scale=0.22]{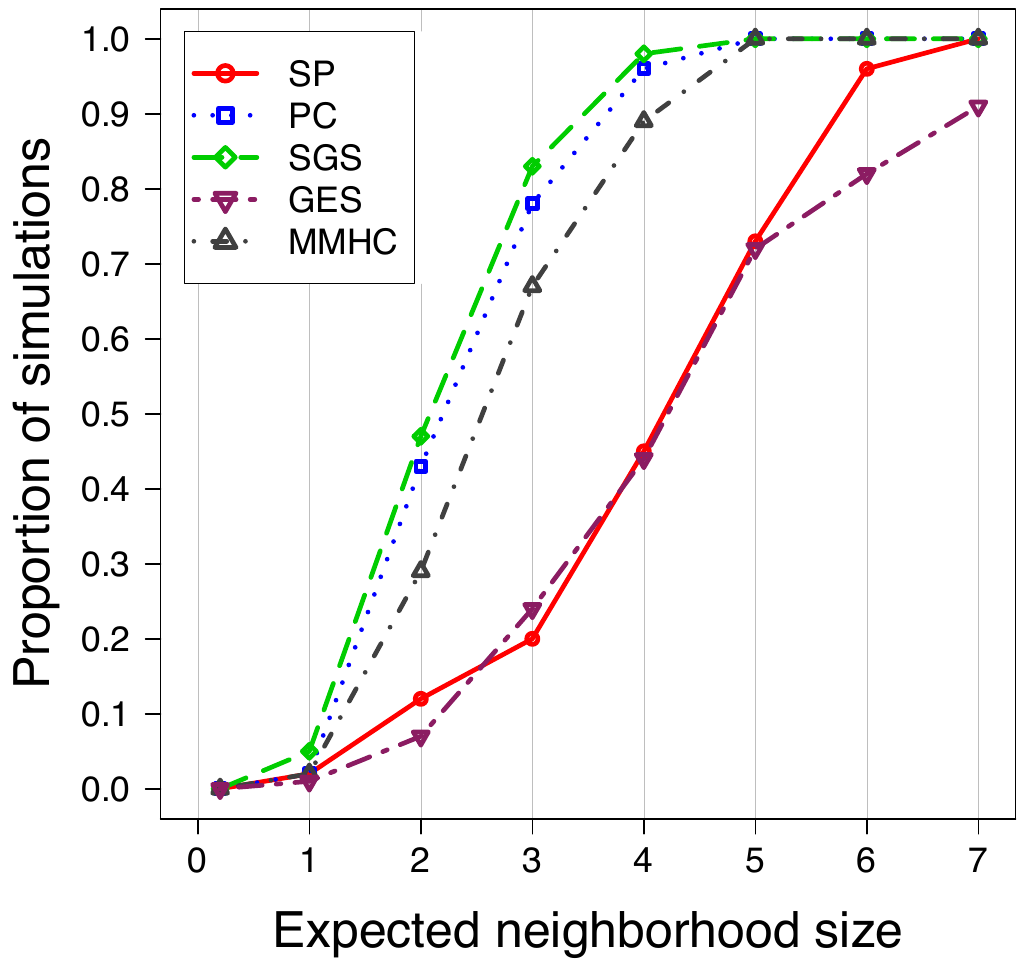}\label{fig:graph_18}}\\
	\caption{Proportion of simulations out of $100$ simulations in which the algorithms outputted a skeleton with missing edges for 8-node random DAG models.}
	\label{fig_miss_edge_p8}
\end{figure}

\clearpage

\bibliographystyle{wb_stat}
\bibliography{Biblio_Causal}

\end{document}